\documentclass[12pt]{amsart}
\usepackage{amssymb}
\usepackage{stmaryrd}
\usepackage{CJK,CJKnumb,CJKulem}
\usepackage{geometry}
\usepackage{fancyhdr}
\usepackage{txfonts}

\RequirePackage{amsmath}
\RequirePackage{amssymb}
\usepackage{amscd,latexsym,amsthm,amsfonts,amssymb,amsmath,amsxtra}
\usepackage[colorlinks=true,urlcolor=blue,bookmarks=true,bookmarksopen=true,citecolor=blue,]{hyperref}
\usepackage[all]{xy}

    \newcommand{\BA}{{\mathbb {A}}} 
    \newcommand{\BC}{{\mathbb {C}}} 
     
    \newcommand{\BG}{{\mathbb {G}}}

    \newcommand{\CA}{{\mathcal {A}}} 
     
    \newcommand{\CE}{{\mathcal {E}}} \newcommand{\CF}{{\mathcal {F}}}
     \newcommand{\CH}{{\mathcal {H}}}
    \newcommand{\CI}{{\mathcal {I}}} \newcommand{\CJ}{{\mathcal {J}}}

     \newcommand{\CP}{{\mathcal {P}}}

     \newcommand{\CZ}{{\mathcal {Z}}}

     \newcommand{\RD}{{\mathrm {D}}}

    \newcommand{\Ad}{{\mathrm{Ad}}}

     \newcommand{\GL}{{\mathrm{GL}}}
    
    \newcommand{\Hom}{{\mathrm{Hom}}}

    \newcommand{\Ind}{{\mathrm{Ind}}}\newcommand{\ind}{{\mathrm{ind}}}

    \newcommand{\Mat}{{\mathrm{Mat}}}

    \newcommand{\PGL}{{\mathrm{PGL}}} 
     \newcommand{\PD}{{\mathrm{PD}}}
    
    \renewcommand{\Re}{{\mathrm{Re}}}

    \newcommand{\Res}{{\mathrm{Res}}}

    \newcommand{\SO}{{\mathrm{SO}}}\newcommand{\Sp}{{\mathrm{Sp}}}

    \newcommand{\tr}{{\mathrm{tr}}}

    \newcommand{\Span}{{\mathrm{Span}}}
     
    \newcommand{\diag}{\mathrm{diag}}

\renewcommand{\Re}{{\mathrm{Re}}}

\newcommand{\ud}{\,\mathrm{d}}

\newcommand{\ol}{\overline}

\newcommand{\lra}{\longrightarrow}

\newcommand{\lto}{\longmapsto}
\newcommand{\bs}{\backslash}

\def\bks{{\backslash}}

\def\diag{{\rm diag}}

\def\eps{{\epsilon}}

\def\lam{{\lambda}}

\def\sig{{\sigma}}

    \newcommand{\wt}{\widetilde}

    \newcommand{\pair}[1]{\langle {#1} \rangle}

    \newcommand{\sk}{\medskip}

    \newcommand{\s}{\sk\noindent}

    \theoremstyle{plain}
    \newtheorem{thm}{Theorem}[section] \newtheorem{cor}[thm]{Corollary} 
    \newtheorem{lem}[thm]{Lemma}  \newtheorem{prop}[thm]{Proposition} \newtheorem{rmk}[thm]{Remark}
      \newtheorem{defn-prop}[thm]{Definition-Proposition}

    \numberwithin{equation}{section}

\makeatletter
\def\legendre@dash#1#2{\hb@xt@#1{%
  \kern-#2\p@
  \cleaders\hbox{\kern.5\p@
    \vrule\@height.2\p@\@depth.2\p@\@width\p@
    \kern.5\p@}\hfil
  \kern-#2\p@
  }}
\def\@legendre#1#2#3#4#5{\mathopen{}\left(
  \sbox\z@{$\genfrac{}{}{0pt}{#1}{#3#4}{#3#5}$}%
  \dimen@=\wd\z@
  \kern-\p@\vcenter{\box0}\kern-\dimen@\vcenter{\legendre@dash\dimen@{#2}}\kern-\p@
  \right)\mathclose{}}
\newcommand\legendre[2]{\mathchoice
  {\@legendre{0}{1}{}{#1}{#2}}
  {\@legendre{1}{.5}{\vphantom{1}}{#1}{#2}}
  {\@legendre{2}{0}{\vphantom{1}}{#1}{#2}}
  {\@legendre{3}{0}{\vphantom{1}}{#1}{#2}}
}
\def\dlegendre{\@legendre{0}{1}{}}
\def\tlegendre{\@legendre{1}{0.5}{\vphantom{1}}}
\makeatother

\begin{document}

\title[Jacquet-Langlands and Twisted Descent]{The Jacquet-Langlands Correspondence via Twisted Descent}

\author{Dihua Jiang}
\address{School of Mathematics, University of Minnesota, Minneapolis, MN 55455, USA}
\email{dhjiang@math.umn.edu}

\author{Baiying Liu}
\address{School of Mathematics, Institute for Advanced Study, Einstein Drive, Princeton, NJ 08540, USA}
\email{liu@ias.edu}

\author{Bin Xu}
\address{Department of Mathematics, Sichuan University, Chengdu, 610064, People's Republic of China}
\email{binxu@scu.edu.cn}

\author{Lei Zhang}
\address{Department of Mathematics, National University of Singapore, Singapore 119076}
\email{matzhlei@nus.edu.sg}

\date{\today}


\thanks{The work of the first named author is supported in part by the NSF Grant DMS--1301567, that of the second
named author is supported in part by
NSF Grant DMS-1302122 and by a postdoc research fund from Department of Mathematics, University of Utah, that of the third named author is
supported in part by National Natural Science Foundation of China (No.11501382), and by the National Key Basic Research Program of China (No. 2013CB834202), and that of the fourth named author is
supported in part by the National University of Singapore’s start up grant}


\begin{abstract}
The existence of the well-known Jacquet-Langlands correspondence
was established by Jacquet and Langlands via the trace formula method in 1970 (\cite{J-L}).
An explicit construction of such a correspondence was obtained by Shimizu via theta series in 1972 (\cite{Sh72}).
In this paper, we extend the automorphic descent method of Ginzburg-Rallis-Soudry (\cite{GRS11}) to a new setting.
As a consequence, we recover the classical Jacquet-Langlands correspondence for $\PGL(2)$ via a new explicit construction.
\end{abstract}

\maketitle


\section{Introduction}
The classical Jacquet-Langlands correspondence between automorphic forms on $\GL(2)$ and $\RD^\times$, with a quaternion division algebra
$\RD$, is one of the first established instances of Langlands functorial transfers. The existence of such a correspondence was
established by Jacquet and Langlands via the trace formula method in 1970 (\cite{J-L}).
An explicit construction of the Jacquet-Langlands correspondence was obtained by Shimizu via theta series in 1972 (\cite{Sh72}).
Shimizu's construction was extended in the general framework of theta correspondences for reductive dual
pairs in the sense of Howe (\cite{H79}), and has been important to many arithmetic applications, including the famous Waldspurger formula for
the central value of the standard $L$-function of $\GL(2)$ (\cite{W85a}),
and the work of Harris and Kudla on Jacquet's conjecture for the central value of
the triple product $L$-function (\cite{HK04}), for instance.

The automorphic descent method developed by Ginzburg, Rallis and Soudry in \cite{GRS11} constructs a map which is backward to the
Langlands functorial transfer from quasi-split classical groups to the general linear group. The starting point of their method is the symmetry of
the irreducible cuspidal automorphic representation on $\GL(n)$. However, their theory is not able to cover the classical groups which are
not quasi-split. In this paper we extend their method by considering more invariance properties of irreducible cuspidal automorphic representations
of $\GL(n)$, so that this extended descent is able to reach certain classical groups which are not quasi-split. Due to the nature of
the newly introduced invariance property and the new setting of the construction, we call the method introduced below
{\it twisted automorphic descent}.

Let $\tau$ be an irreducible unitary cuspidal automorphic representation of $\GL_{2n}(\BA)$, where $\BA$ is the ring of adeles of a number
field $F$. Assume that the partial exterior square $L$-function $L^S(s,\tau,\wedge^2)$ has a simple pole at $s=1$.
It is well known now that $\tau$ is the Langlands functorial transfer from an irreducible generic cuspidal automorphic representation $\pi_0$ of
the $F$-split odd special orthogonal group $G_0(\BA)=\SO_{2n+1}(\BA)$. The automorphic descent method of Ginzburg, Rallis and Soudry
(\cite{GRS11}) and the irreducibility of the descent (\cite{JngS03}) show that this $\pi_0$ is unique and can be explicitly constructed by means of
a certain Bessel-Fourier coefficients of the residual representation of $\SO_{4n}(\BA)$ with cuspidal datum $(\GL_{2n},\tau)$.

The objective of this paper is to extend the descent method of \cite{GRS11} to construct more general cuspidal representations and more
general groups which are pure inner forms of the $F$-split odd special orthogonal group $G_0(\BA)$. To this end, we
take $\sigma$ to be an irreducible cuspidal automorphic representation of
an $F$-anisotropic $\SO_2^\delta$ associated to a non-square class $\delta$ of $F^\times$, and assume that the central value $L(\frac{1}{2},\tau\times\sigma)$ is nonzero. The main idea
is to make the conditions such as $L(\frac{1}{2},\tau\times\sigma)\neq 0$ into play in the construction of more general cuspidal automorphic
representations of classical groups using the irreducible unitary cuspidal automorphic representation $\tau$ of $\GL$ and the irreducible cuspidal representation $\sigma$ of $\SO_2^{\delta}$. We refer to
\cite{JZ15} for a more general framework of such constructions, which are technically much more involved than the current case in this paper. Hence
we leave to \cite{JZ15} the detailed discussions for general construction.
We give below more detailed description of the construction and the main results of the paper.

\subsection{A residual representation of $\SO_{4n+2}^{\delta}(\BA)$}
Let $V$ be a quadratic space of dimension $4n+2$ defined over $F$
with a non-degenerate quadratic form $\pair{\cdot , \cdot }$.
We assume that the Witt index of $V$ is $2n$ with a polar decomposition
$$
V=V^+\oplus V_0\oplus V^-,
$$
where $V^+$ is a maximal isotropic subspace of $V$, and $V_0$ is an anisotropic subspace of dimension $2$.
We may take the quadratic form on $V_0$ to be associated to $\displaystyle{J_\delta=\begin{pmatrix}1&0\\0&\delta\end{pmatrix}}$,
where
$-\delta\notin {F^\times}^2$,
and the quadratic form of $V$ is taken to be associated to
$$\begin{pmatrix}
   &&w_{2n}\\
   &J_\delta&\\
   w_{2n}&&
  \end{pmatrix},
$$
where $w_r$ is the ($r\times r$)-matrix with $1$'s on its anti-diagonal and zero elsewhere.
Denote by $H^\delta=\SO_{4n+2}^\delta$ the corresponding $F$-quasisplit special even orthogonal group. We fix a maximal flag
$$
\CF :\quad 0\subset V_1^+\subset V_2^+\subset \cdots \subset V_{2n}^+=V^+
$$
in $V^+$, and choose a basis
$\{e_1, \cdots, e_{2n}\}$ of $V^+$ over $F$ such that $V_i^+=\Span\{e_1,\cdots, e_i\}$. We also let $\{e_{-1},\cdots, e_{-2n}\}$
be a basis for $V^-$, which is dual to $\{e_1, \cdots, e_{2n}\}$ in the sense that
$$
\pair{e_i,e_{-j}}=\delta_{i,j}\ \text{for}\ 1\leq i,j\leq 2n,
$$
and let $V_\ell^- = \Span\{e_{-1}, \ldots, e_{-\ell}\}$.

Let $P$ be the parabolic subgroup fixing $V^+$.
Then $P$ has a Levi decomposition $P=MU$ such that $M\simeq \GL_{2n}\times\SO^\delta_2$, where $\SO^\delta_2$ is the $F$-quasisplit
special orthogonal group of $(V_0,J_\delta)$.
Let $\tau$ be an irreducible unitary cuspidal automorphic representation of $\GL_{2n}(\BA)$, and $\sigma$ an irreducible unitary (cuspidal) representation of $\SO^\delta_2(\BA)$.  Then $\tau\otimes \sigma$
is an irreducible unitary cuspidal representation of $M(\BA)$. For $s\in \BC$ and an automorphic function
$$
\phi_{\tau\otimes \sigma}\in \CA(M(F)U(\BA)\bs H^\delta(\BA))_{\tau\otimes \sigma},
$$
following \cite[\S II.1]{MW95}, one defines
$\lambda_s \phi_{\tau\otimes \sigma}$ to be $(\lambda_s\circ m_P)\phi_{\tau\otimes \sigma}$, where $\lambda_s \in X_{M}^{H^\delta}\simeq \BC$
(see \cite[\S I.1]{MW95} for the definition of $X_M^{H^\delta}$ and the map $m_P$), and defines the corresponding Eisenstein series by
$$
E(h, s,\phi_{\tau\otimes \sigma})=\sum_{\gamma\in P(F)\bs H^\delta(F)}\lambda_s \phi_{\tau\otimes \sigma}(\gamma h),
$$
which converges absolutely for $\Re(s)\gg 0$ and has meromorphic continuation to the whole complex plane (\cite[\S IV]{MW95}).
We note here that under the the normalization of Shahidi, we can take $\lambda_s=s\wt \alpha$ (see \cite{Sh90}), where $\alpha$ is the unique reduced root of the maximal $F$-split torus of $H^\delta$ in $U$.

As in \cite{JLZ13}, define
$$
\beta(s)=L(s+1,\tau\times \sigma)\cdot L(2s+1,\tau,\wedge^2),
$$
where the $L$-functions are defined through the Langlands-Shahidi method (\cite{Sh10}). Then define the normalized Eisenstein series
$$
E^*(h, s,\phi_{\tau\otimes \sigma})=\beta(s)E(h, s,\phi_{\tau\otimes \sigma}),
$$
which has the following properties.

\begin{prop}[\cite{JLZ13}, Proposition 4.1]\label{jlz13}
 Let $\tau$ and $\sigma$ be as above. Then the
 normalized Eisenstein series $E^*(h, s,\phi_{\tau\otimes \sigma})$ has a simple pole at $s=\frac{1}{2}$ if and only if
 $L(s, \tau, \wedge^2)$ has a pole at $s=1$ and $L(\frac{1}{2}, \tau\times \sigma)\neq 0$.
\end{prop}

Let $\CE_{\tau\otimes \sigma}$ denote the automorphic representation of $H^\delta(\BA)$ generated by the residues at $s=\frac{1}{2}$
of $E(h,s,\phi_{\tau\otimes \sigma})$
for all $\phi_{\tau\otimes \sigma}\in \CA(M(F)U(\BA)\bs H^\delta(\BA))_{\tau\otimes \sigma}$. From now on, we assume that
$L(s, \tau, \wedge^2)$ has a pole at $s=1$ and $L(\frac{1}{2}, \tau\times \sigma)\neq 0$. In this case, $\tau$ has trivial central
character (\cite{JS90}).
By Proposition \ref{jlz13}, the residual
representation $\CE_{\tau\otimes \sigma}$ is nonzero.
By the $L^2$-criterion in \cite{MW95}, \cite[Theorem 6.1]{JLZ13} shows that the residual representation $\CE_{\tau\otimes \sigma}$
is square-integrable.
Moreover, the residual representation $\CE_{\tau\otimes \sigma}$ is irreducible, following Theorem A of Moeglin in
\cite{M11}, for instance.
Note that the global Arthur parameter (see \cite{A}) for $\CE_{\tau\otimes \sigma}$ is
$(\tau,2) \boxplus \psi_\sigma$ (\cite[\S 6]{JLZ13}).

\subsection{Fourier coefficients attached to partition $[(2\ell+1)1^{4n-2\ell+1}]$}
Following \cite{J14}, one defines Fourier coefficients of automorphic forms of classical groups attached to nilpotent orbits and hence to
a partition. For $H^\delta=\SO^\delta_{4n+2}$, we consider here the Fourier coefficients attached to the partition $[(2\ell+1)1^{4n-2\ell+1}]$ of
$4n+2$ with $1\leq \ell \leq 2n$, which is also known as Bessel-Fourier coefficients. More precisely,
for $1\leq \ell \leq 2n$, consider the following partial flag
$$
\CF_\ell: \quad 0\subset V_1^+\subset V_2^+\subset\cdots \subset V_\ell^+,
$$
which defines the standard parabolic subgroup $P_\ell$ of $H^\delta$, with the Levi decomposition $P_\ell=M_\ell\cdot N_\ell$.
The Levi part $M_\ell\simeq (\GL_1)^\ell\times \SO(W_\ell)$, with
$$
W_\ell= (V_\ell^+\oplus V_\ell^-)^\perp.
$$

Following \cite{W01}, the $F$-rational nilpotent orbits in the $F$-stable nilpotent orbits in the Lie algebra of $\SO^\delta_{4n+2}$ associated to
the partition $[(2\ell+1)1^{4n-2\ell+1}]$ are parameterized by $F$-rational orbits in the $F$-anisotropic vectors in $W_\ell$ under the action of
$\GL_1(F)\times\SO(W_\ell)$. The action of $\GL_1(F)\times\SO(W_\ell)$ on $W_\ell$ is induced naturally from the adjoint action of $M_\ell$ on
the unipotent radical $N_\ell$.

Take an anisotropic vector $w_0\in W_\ell$ with $\pair{w_0,w_0}$ in a given square class of $F^\times$, and define a homomorphism
$\chi_{\ell,w_0}: N_\ell\lra \BG_a$ by
$$
\chi_{\ell,w_0}(u)=
\sum_{i=2}^\ell \pair{u\cdot e_i,e_{-(i-1)}}+\pair{u \cdot w_0,e_{-\ell}}.
$$
Recall that here $\pair{\cdot ,\cdot }$ denotes the quadratic form on $V$. Define also a character
$$
\psi_{\ell,w_0}=\psi\circ\chi_{\ell,w_0}: N_\ell(\BA)\lra \BC^\times,
$$
where $\psi: F\bs \BA\lra \BC^\times$ is a fixed non-trivial additive character. Hence the character $\psi_{\ell,w_0}$ is trivial on $N_{\ell}(F)$.
Now the adjoint action of $M_\ell$ on $N_\ell$ induces an action of $\SO(W_\ell)$ on the set of all such characters $\psi_{\ell,w_0}$.
The stabilizer $L_{\ell,w_0}$ of $\chi_{\ell, w_0}$ in $\SO(W_\ell)$ equals to $\SO(w_0^\perp\cap W_\ell)$.

Let $\Pi$ be an automorphic representation of $H^\delta(\BA)$ (see \cite[Section 4.6]{BJ79}), occurring in the discrete spectrum. 
For $f\in V_\Pi$ and $h\in H^\delta(\BA)$, we define the
$\psi_{\ell,w_0}$-Fourier coefficients of $f$ by
\begin{equation}\label{Gelfand-Graev}
 f^{\psi_{\ell, w_0}}(h)=\int_{[N_\ell]} f(vh)\psi_{\ell,w_0}^{-1}(v)\ \mathrm{d}v,
\end{equation}
here $[N_\ell]$ denotes the quotient $N_\ell(F)\bs N_\ell(\BA)$. This is one of the Fourier coefficients of $f$ associated to the partition
$[(2\ell+1)1^{4n-2\ell+1}]$. It is clear that $f^{\psi_{\ell, w_0}}(h)$ is left $L_{\ell,w_0}(F)$-invariant, with
$L_{\ell,w_0}=\SO(w_0^\perp\cap W_\ell)$. Following \cite[Section 3.1]{GRS11}, we define the space
$$
\sigma_{\psi_{\ell,w_0}}(\Pi)=L_{\ell,w_0}(\BA)-\text{Span} \left\{f^{\psi_{\ell, w_0}}\bigg|_{L_{\ell,w_0}(\BA)}\ \bigg|\ f\in V_\Pi\right\},
$$
which is a representation of $L_{\ell,w_0}(\BA)$, with right translation action. 
The situation that is the main concern of the paper is the case of the so-called first occurrence, as described in Theorem \ref{Main theorem of descent}. In such a situation, the function $f^{\psi_{\ell, w_0}}(h)$ will be cuspidal for all $f\in V_{\Pi}$, in particular, square-integrable.
Hence we may identify the space $\sigma_{\psi_{\ell,w_0}}(\Pi)$ as a subspace of the discrete spectrum of the space of square-integrable
automorphic functions on $L_{\ell,w_0}(\BA)$, and $\sigma_{\psi_{\ell,w_0}}(\Pi)$ becomes an automorphic $L_{\ell,w_0}(\BA)$-module in the sense of \cite[Section 4.6]{BJ79}. 

To make a precise choice of the anisotropic vector $w_0$ for the cases $1\leq \ell <2n$, we take
$$
w_0=y_\alpha=e_{2n}-\frac{\alpha}{2}e_{-2n}
$$
for $\alpha\in F^\times$, and hence we have $\pair{w_0,w_0}=-\alpha$. For such an $\alpha$, we consider the three-dimensional quadratic
form given by
$J_{\delta,\alpha}=\begin{pmatrix}1&&\\&\delta&\\&&\alpha\end{pmatrix}$,
which can be split or non-split over $F$, depending on the Hilbert symbol $(\delta,\alpha)$.
For the cases $\ell=2n$, we take any nonzero $w_0\in V_0$.
Hence for $1\leq \ell < 2n$ and $w_0=y_\alpha$, we have
$
L_{\ell,\alpha}=\SO^{\delta,\alpha}_{4n-2\ell+1},
$
which is a special odd orthogonal group defined by the form
$$
J=\begin{pmatrix}
   &&w_{2n-\ell-1}\\&J_{\delta,\alpha}&\\w_{2n-\ell-1}&&
  \end{pmatrix}.
$$
We denote $\psi_{\ell,y_\alpha}$ by $\psi_{\ell,\alpha}$,
and $\sigma_{\psi_{\ell,y_\alpha}}(\Pi)$ by $\sigma_{\psi_{\ell,\alpha}}(\Pi)$.

\subsection{The twisted automorphic descent}
Now we apply the $\psi_{\ell,w_0}$-Fourier coefficients to the residual representation $\CE_{\tau\otimes\sigma}$ of
$\SO_{4n+2}^\delta(\BA)$ and investigate more carefully its properties depending on the integer $\ell$.
One of the main results in this paper is the following theorem.

\begin{thm}\label{Main theorem of descent}
Assume that an irreducible unitary cuspidal automorphic representation $\tau$ of $\GL_{2n}(\BA)$ has the property that
$L(s, \tau, \wedge^2)$ has a pole at $s=1$ and there exists an irreducible unitary (cuspidal) representation $\sigma$ of $\SO^\delta_2(\BA)$
such that $L(\frac{1}{2}, \tau\times \sigma)\neq 0$. Then the following hold.
\begin{enumerate}
\item The representation $\sigma_{\psi_{\ell,\alpha}}(\CE_{\tau\otimes\sigma})$ of $\SO^{\delta,\alpha}_{4n-2\ell+1}(\BA)$ is zero
for all $n<\ell\leq 2n$.
\item For any square class $\alpha$ in $F^\times$, the representation $\sigma_{\psi_{n,\alpha}}(\CE_{\tau\otimes\sigma})$ of $\SO^{\delta,\alpha}_{2n+1}(\BA)$ is cuspidal automorphic.
\item There exists a square class $\alpha$ in $F^\times$ such that the representation $\sigma_{\psi_{n,\alpha}}(\CE_{\tau\otimes\sigma})$ of $\SO^{\delta,\alpha}_{2n+1}(\BA)$ is nonzero, and in this case
    $$
    \sigma_{\psi_{n,\alpha}}(\CE_{\tau\otimes\sigma})=\pi_1\oplus\pi_2\oplus\cdots\oplus\pi_r\oplus\cdots,
    $$
    where $\pi_i$ are irreducible cuspidal automorphic representations of $\SO^{\delta,\alpha}_{2n+1}(\BA)$, which are nearly equivalent, but
    are not globally equivalent, i.e. the decomposition is multiplicity free.
\item When $\sigma_{\psi_{n,\alpha}}(\CE_{\tau\otimes\sigma})$ is nonzero, any direct summand $\pi_i$ of $\sigma_{\psi_{n,\alpha}}(\CE_{\tau\otimes\sigma})$ has a weak Langlands functorial
transfer to $\tau$ in the sense that the Satake parameter of the local unramified component $\tau_v$ of $\tau$ is the local functorial transfer
of that of the local unramified component $\pi_v$ of $\pi$ for almost all unramified local places $v$ of $F$.
\item When $\sigma_{\psi_{n,\alpha}}(\CE_{\tau\otimes\sigma})$ is nonzero, every irreducible
direct summand of $\sigma_{\psi_{n,\alpha}}(\CE_{\tau\otimes\sigma})$ has a nonzero Fourier coefficient attached to the
partition $[(2n-1)1^2]$.
\item The residual representation $\CE_{\tau\otimes\sigma}$ has a nonzero Fourier coefficient attached to the partition
$[(2n+1)(2n-1)1^2]$.
\end{enumerate}
\end{thm}

For these square classes $\alpha$ in $F^\times$ such that $\sigma_{\psi_{n,\alpha}}(\CE_{\tau\otimes\sigma})$ is nonzero,
we call $\sigma_{\psi_{n,\alpha}}(\CE_{\tau\otimes\sigma})$ the {\sl twisted automorphic descent} of $\tau$ to
$\SO^{\delta,\alpha}_{2n+1}(\BA)$. It is clear from the notation and more importantly the explicit construction that
the twisted automorphic descent $\sigma_{\psi_{n,\alpha}}(\CE_{\tau\otimes\sigma})$ depends on the data $(\tau; \delta,\sigma; \alpha)$.
The $\GL_1(F)\times\SO(W_\ell)(F)$-orbits of anisotropic vectors in $W_\ell$ labeled by the pair $(\delta, \alpha)$
determine more refined properties of the
twisted automorphic descent $\sigma_{\psi_{n,\alpha}}(\CE_{\tau\otimes\sigma})$, which will be briefly discussed at the end of Section 3
in connection to Parts (5) and (6) of Theorem \ref{Main theorem of descent} and to the structure of the relevant global Vogan packets. We refer to
\cite{JZ15} for more discussion on the general framework of the twisted automorphic descent constructions for classical groups with
connections to explicit structures of global Arthur packets and global Vogan packets. In order to show the potential of
the theory of twisted automorphic descents, we provide in Section 5 a complete description of the method and result for
the case of $n=1$, which recovers the classical Jacquet-Langlands correspondence. Hence the twisted automorphic descent discussed in this
paper provides a new method to establish the Jacquet-Langlands correspondence for $\PGL(2)$.

Parts (1) and (2) of the theorem are usually called the {\sl tower property} and {\sl cuspidality}.
They will be proved in Section 2 after establishing
the vanishing property of the corresponding local Jacquet module at one unramified local place, following the idea of \cite{GRS11}. In order
to establish the {\sl first occurrence} at $\ell=n$, which is the first assertion in Part (3) of the theorem,
we have to understand the structure of the Fourier coefficients
of the residual representation $\CE_{\tau\otimes\sigma}$ following the general framework of \cite{J14}. We first prove that
the residual
representation $\CE_{\tau\otimes\sigma}$ has a nonzero Fourier coefficient attached to the partition $[(2n)^21^2]$ in Section 3
(Proposition 3.1),
which is based on the cuspidal support of $\CE_{\tau\otimes\sigma}$. After that,
we use the argument similar to that in \cite[Section 5]{GJR05}, but with full detail, to prove that there exists at least
one square class $\alpha$ in $F^\times$ such
that the twisted automorphic descent $\sigma_{\psi_{n,\alpha}}(\CE_{\tau\otimes\sigma})$ is nonzero (Proposition~\ref{thm2}),
which yields the first assertion in Part (3) of the theorem. It is clear that the second assertion in Part (3) of
the theorem follows from Part (4) for the assertion of nearly equivalence, while the multiplicity free decomposition follows from
the local uniqueness of Bessel models (\cite{AGRS}, \cite{JSZ10} and \cite{GGP12}). If we assume that the global Gan-Gross-Prasad conjecture
holds for this case, then may deduce that $\sigma_{\psi_{n,\alpha}}(\CE_{\tau\otimes\sigma})$ is irreducible. We will not discuss this issue
in such a generality in this paper, but we will discuss the case of $n=1$ in Section 5.
Next, we continue in Section 3 to obtain Part (5) of the theorem by showing that the irreducible summands in the
twisted automorphic descent $\sigma_{\psi_{n,\alpha}}(\CE_{\tau\otimes\sigma})$ have nonzero Fourier coefficients attached to the partition
$[(2n-1)1^2]$. Note that if $n=1$, this partition is trivial, and the corresponding Fourier coefficient is the identity map.
It is clear that if the group $\SO^{\delta,\alpha}_{3}$ is
$F$-anisotropic, then its only nilpotent orbit is the trivial one, corresponding to the trivial partition.
However, if the group $\SO^{\delta,\alpha}_{3}$ is $F$-split, because of Part (4), we expect that the
twisted automorphic descent has a nonzero Whittaker-Fourier coefficient. There will be a more detailed discussion for this case in Section 5.
In Section 4, we are going to prove that the Satake parameter of the unramified local component of any of the direct summand $\pi_i$ of
$\sigma_{\psi_{n,\alpha}}(\CE_{\tau\otimes\sigma})$ transfers canonically to the Satake parameter of the unramified local component of
$\tau$, and hence the direct summands of $\sigma_{\psi_{n,\alpha}}(\CE_{\tau\otimes\sigma})$ are nearly equivalent. This proves Parts (3) and (4) of
the theorem. The last part (Part (6) of the theorem) follows from Part (5) as consequence of the
Fourier coefficients associated to a composite of two partitions as discussed in \cite{GRS03} and \cite{JL15}, which will be
briefly discussed before the end of Section 3.

Acknowledgement. This material is based upon work supported by the National Science Foundation under agreement No. DMS-1128155. Any opinions, findings and conclusions or recommendations expressed in this material are those of the authors and do not necessarily reflect the views of the National Science Foundation. 
We also would like to thank the referee for very helpful comments on the previous version of the paper, which make the paper more readable.


\section{On Vanishing and Cuspidality}


We are going to show that the $\psi_{\ell,w_0}$-Fourier coefficients of the residual representation $\CE_{\tau\otimes\sigma}$ of
$\SO_{4n+2}^\delta(\BA)$ vanishes for all $n<\ell\leq 2n$. Following the tower property proved in \cite{GRS11}, the twisted automorphic
descent $\sigma_{\psi_{n,\alpha}}(\CE_{\tau\otimes\sigma})$ is cuspidal, which might also vanish. Such vanishing property of the
family of the Fourier coefficients $\CE_{\tau\otimes\sigma}^{\psi_{\ell,w_0}}$ should be determined by the global Arthur parameter of
the square-integrable residual representation $\CE_{\tau\otimes\sigma}$ and the structure of its Fourier coefficients as conjectured in
\cite{J14}. Through the local-global relation, one may take the local approach to prove such vanishing property as shown in \cite{GRS11}.
However, a purely global argument should be interesting and expected (\cite{GRS99a}). We will come back to this issue in future.

In the following, we follow \cite{GRS11} to take the local approach, which is a calculation of twisted Jacquet module at one unramified local
place. This proves Part (1) of Theorem \ref{Main theorem of descent}. Following the argument of the tower property in \cite{GRS11},
we deduce in the second subsection the cuspidality of the twisted automorphic descent, which is Part (2) of
Theorem \ref{Main theorem of descent}.

\subsection{Calculation of certain Jacquet modules}
Most of the results are deduced from the general results of \cite[Chapter 5]{GRS11}.

Let $k$ be a non-Archimedean local field and $V$ be a non-degenerate quadratic space of dimension $4n+2$ over
$k$. We have a polar decomposition
$$V=V^+\oplus V_0 \oplus V^-,$$
where $V^+$ is a maximal isotropic subspace of $V$ and $V_0$ is anisotropic. There are two cases to be considered:
\begin{enumerate}
 \item $\dim V_0=0$, if $\SO(V)$ is split;
 \item  $\dim V_0=2$, if $\SO(V)$ is quasi-split.
\end{enumerate}
We denote the Witt index of $V$ by $\wt m$, which can equal to $2n+1$ or $2n$ in case (1) or (2), respectively.

Let $H=\SO(V)$ be the special orthogonal group on $V$ over
$k$. Fix a basis $\{e_1,\cdots , e_{2n}\}$ of $V^+$, and if $V_0\neq 0$, we take a basis $\{e_0^{(1)}, e_0^{(2)}\}$ for $V_0$ such that $\pair{e_0^{(1)}, e_0^{(1)}}=1$ and $\pair{e_0^{(2)},e_0^{(2)}}=\delta$ for some $\delta\in k^\times$. Otherwise we fix a basis $\{e_1,\cdots , e_{2n+1}\}$ for $V^+$. As in the global case (\S 2.2),
for $0<\ell\leq 2n$, let $P_\ell$ be the parabolic subgroup of $H$ preserving the partial flag
$$
\CF_\ell: 0\subset V_1^+\subset V_2^+\subset\cdots \subset V_\ell^+.
$$
Then $P_\ell=M_\ell\cdot N_\ell$ with $M_\ell\simeq \BG_m^\ell\times \SO(W_\ell)$,
where $W_\ell$ is the same as in \S 1.2.
For $\alpha\in k^\times$, we take an anisotropic vector $w_0\in V$ as follows:
\begin{equation}\label{local anisotropic vector}
w_0=\begin{cases}
     y_\alpha=e_{2n}-\frac{\alpha}{2}e_{-2n}, & \text{if $\ell < 2n$};\\
     &\\
     \text{$\alpha e_0^{(1)}$ or $\alpha e_0^{(2)}$ in $V_0$}, & \text{if $\ell=2n$ and $V_0\neq 0$}.
    \end{cases}
\end{equation}

\begin{rmk}\label{rmk on anisotropic vector}
By a change of basis if necessary, we see that each anisotropic vector in $V$ is of the form
$\alpha e_i+\beta e_{-i}$ or $\alpha e_0^{(j)}$ (this occurs only if $V_0\neq 0$) for some $\alpha,\beta\in k^\times$. Then the above choices of $w_0$ essentially represent any anisotropic vector in $V$.
\end{rmk}
Similar to the global case (\S 2.2), one defines a character $\psi_\ell$ on $N_\ell(k)$ by
$\psi_\ell=\psi\circ \chi_{\ell,w_0}$, where $\psi$ is a non-trivial additive
character on $k$.
Let $L_{\ell,w_0}=\SO(W_\ell\cap w_0^\perp)$ be the stabilizer of $\chi_{\ell,w_0}$
in $\SO(W_\ell)$.

For any smooth representation $(\Pi,V_\Pi)$ of $H(k)$, we define the (twisted)
Jacquet module
$$J_{\psi_{\ell,w_0}}(V_{\Pi})=V_\Pi/\Span\{\Pi(u)\xi-\psi_{\ell, w_0}(u)\xi\ |\ u\in N_\ell(k),\  \xi\in V_\Pi\},$$
which is an $L_{\ell,w_0}$-module.
To simplify the notation, we will use $\chi_{\ell,\alpha}$, $L_{\ell,\alpha}$ and $J_{\psi_\ell,\alpha}$ if $w_0=y_\alpha$.
Note that with the notation of the previous section, for $0\leq \ell<2n$,
the form of the odd orthogonal group $L_{\ell,\alpha}$ is
$$J=\begin{pmatrix}
   &&w_{2n-\ell-1}\\&J_{\delta,\alpha}&\\w_{2n-\ell-1}&&
  \end{pmatrix},
$$
and hence we have the following cases:
\begin{itemize}
 \item[(i)] if $J_{\delta,\alpha}$ is non-split over $k$, then $L_{\ell,\alpha}$ is non-split, and $\mathrm{Witt}(k\cdot y_{-\alpha}+V_0)=0$;
 \item[(ii)] otherwise, the group $L_{\ell,\alpha}$ is $k$-split. In this case, $\mathrm{Witt}(k\cdot y_{-\alpha}+V_0)=1$ if $V_0\neq 0$ (i.e. $-\delta\notin {k^\times}^2$), and
 $\mathrm{Witt}(k\cdot y_{-\alpha}+V_0)=0$ if $V_0=0$.
\end{itemize}
Moreover, if $\ell=2n$, this will be a degenerate case, and $L_{2n,w_0}$ is a trivial group.

For any $0<j\leq 2n$, let
$V_j^+=\Span\{e_1,\cdots, e_j\}$, and let $Q_j$ be the standard parabolic subgroup
of $H$ which preserves $V_j^+$. The group $Q_j$ has a Levi decomposition $Q_j=D_j\cdot U_j$ with
$D_j\simeq \GL_j(k)\times \SO(W_j)$. For $0\leq t <j$, let $\tau^{(t)}$ is the $t$-th Bernstein-Zelevinsky derivative of $\tau$ along
the subgroup
$$
Z_t'=\left\{\begin{pmatrix}
                           I_s&y\\0&z
                          \end{pmatrix}
\in \GL_j(k)\ \big|\ z\in Z_t(k)\right\}
$$
and corresponding to the character
$$
\psi_t'\left(
\begin{pmatrix}
I_s&y\\0&z
\end{pmatrix}
\right)=\psi^{-1}(z_{1,2}+\cdots+z_{t-1,t}).
$$
Then $\tau^{(t)}$ is the representation of $\GL_s(k)$ with $s=j-t$, acting on the Jacquet module $J_{Z_t',\psi_t'}(V_\tau)$ via
the embedding
$$
d\lto \diag(d,I_t)\in \GL_j(k).
$$
For $a\in k^\times$, we also consider the character $\psi_{t,a}^{''}$ on $Z_t'$ defined by
$$
\psi_{t,a}^{''}(
\begin{pmatrix}
 I_s&y\\0&z
\end{pmatrix})=\psi^{-1}(z_{1,2}+\cdots+z_{t-1,t}+ay_{s,1}).
$$
Denote the corresponding Jacquet module $J_{Z_t',\psi_{t,a}^{''}}(V_\tau)$ by $\tau_{(t),a}$, which is a
representation of the mirabolic subgroup $P_{s-1}^1$ of $\GL_s(k)$. Since $\tau_{(t),a}\simeq \tau_{(t),a'}$ for any $a,a'\in k^\times$,
according to \cite[Lemma 5.2]{GRS11}, we sometimes
denote by $\tau_{(t)}$ any of such representations $\tau_{(t),a}$.

Let $\tau$ and $\sigma$ be smooth representations of $\GL_j(k)$ and $\SO(W_j)$ respectively.
The Jacquet module
$J_{\psi_{\ell, \alpha}}(\Ind_{Q_j}^H \tau\otimes \sigma)$
has been studied in \cite{GRS11}. We state it here for completeness.

\begin{prop}[Theorem 5.1 of \cite{GRS11}]\label{Jacquet module 1}
Set $|\cdot|:=|\det(\cdot)|$. The following hold.
\begin{enumerate}
\item Assume that $0\leq \ell <\wt m$ and $1\leq j <\wt m$, then
$$
\begin{aligned} J_{\psi_{\ell, \alpha}}(\Ind_{Q_j}^H \tau\otimes \sigma)\equiv
&\bigoplus_{\ell+j-\wt m<t\leq \ell,\ 0\leq t \leq j}
\ind_{Q_{j-t}'}^{L_{\ell,\alpha}}|\cdot|^{\frac{1-t}{2}}\tau^{(t)}\otimes
J_{\psi'_{{\ell-t},\alpha}}(\sigma^{\omega_b^t})\\
&\oplus \begin{cases}
            \ind_{Q_{j-\ell,\ell}^{'}}^{L_{\ell,\alpha}}|\cdot|^{-\frac{\ell}{2}}\tau_{(\ell)}\otimes \sigma^{\omega_b^{\ell}}, & \text{if $\ell<j$},\\
            0, & \text{otherwise};
           \end{cases}\\
&\oplus \delta_{\alpha}\cdot
\begin{cases}
 \ind_{Q_\alpha'}^{L_{\ell,\alpha}}|\cdot|^{\frac{1+2n-\ell-j}{2}}\tau^{(\ell+j-\wt m)}\otimes J_{\psi'_{{\wt m-j},v_\alpha}}(\sigma^{\omega_b^{\ell+j-2n}}), & \text{if $0<2n-\ell \leq j$},\\
 0, & \text{otherwise}.
\end{cases}
\end{aligned}
$$
\item Assume that $0\leq \ell \leq \wt m$ and $j=\wt m$, then
 $$
 J_{\psi_{\ell, \alpha}}(\Ind_{Q_j}^H \tau\otimes \sigma)\equiv
 \ind_{Q_{\wt m-\ell}'}^{L_{\ell,\alpha}}|\cdot|^{-\frac{\ell}{2}}\tau_{(\ell)}\otimes \sigma^{\omega_b^{\ell}}\oplus
 \delta_\alpha\cdot \ind_{Q_\alpha'}^{L_{\ell,\alpha}}|\cdot|^{\frac{1-\ell}{2}}\tau^{(\ell)}\otimes J_{\psi'_{0,v_\alpha}}(\sigma^{\omega_b^\ell}).
 $$
  \item Assume that $\ell=\wt m$ and $w_0\in V_0$, then
 $$J_{\psi_{\ell,w_0}}(\Ind_{Q_j}^H \tau\otimes \sigma)\equiv d_\tau\cdot J_{\psi'_{{\ell-j},w_0}}(\sigma^{\omega_b^j}),$$
 here $d_\tau=\dim \tau^{(j)}$.
\end{enumerate}
Here $Q_{s}'=L_{\ell,\alpha}\cap \eta_t^{-1}Q_{\ell,j}^{(w)}\eta_t$ with $w$ being the representative of $Q_j\bs H/Q_\ell$ corresponding to the pair $(0,s)$ as in \cite[Chapter 4]{GRS11}, and with $t=j-s$,
$$
Q_{\ell,j}^{(w)}=Q_\ell\cap w^{-1}Q_{j}w, \quad
\eta_t=\begin{pmatrix}
 \epsilon&&\\&I_{4n+2-2\ell}&\\&&\epsilon^*
\end{pmatrix}, \quad \text{and}\
\epsilon=
\begin{pmatrix}
 &I_{\ell-t}\\I_t&
\end{pmatrix}.
$$
Also $Q_{s,t}'=L_{\ell,\alpha}\cap \eta_{s,t}^{-1}Q_{\ell,j}^{(w)}\eta_{s,t}$ with $t$ being the same as above, and
$$\eta_{s,t}=\begin{pmatrix}
 \epsilon&&\\&\gamma_s&\\&&\epsilon^*
\end{pmatrix}, \quad
\gamma_s=
\begin{pmatrix}
 0&I_s&&&\\I_{2n-\ell-s}&0&&&\\&&I_{V_0}&&\\&&&0&I_{2n-\ell-s}\\&&&I_s&0
\end{pmatrix}.$$
Set $\omega_b=\diag\left(I_{\wt m},
\omega_b^0,
I_{\wt m}\right)$, where
$\omega_b^0=
\begin{pmatrix}
1&0\\0&-1
\end{pmatrix}$ if $\wt m=2n$, and
$\omega_b^0=\begin{pmatrix}
1&0\\0&-1
\end{pmatrix}$ if $\wt m=2n+1$.
Moreover, $\delta_\alpha=0$ unless $\mathrm{Witt}(k\cdot y_{-\alpha}+V_0)=1$, and
in such cases, set $v_\alpha\in V_0$ such that $\pair{v_\alpha,v_\alpha}=-\alpha$,
$$\gamma_\alpha=
\begin{pmatrix}
I_{2n-\ell-1} &&&&\\
&1&&&\\
&-v_\alpha&I_{V_0}&&\\
&\frac{\alpha}{2}&v_\alpha'&1&\\
&&&&I_{2n-\ell-1}
\end{pmatrix},
$$
$\gamma=\diag(I_{s-1},\gamma_\alpha, I_{s-1})$, and $\eta_{\gamma_\alpha,t}$ is of the same form as $\eta_{s,t}$ where $\gamma_s$ is replaced by $\gamma_\alpha$.
Finally, $Q_\alpha'=L_{\ell,\alpha}\cap \eta_{\gamma_\alpha,t}^{-1}Q_{\ell,j}^{(w)}\eta_{\gamma_\alpha, t}$, and
$\psi'_{l, \alpha}$ (or $\psi'_{l.v_\alpha}$) denotes the corresponding character like $\psi_{\ell, \alpha}$ but on the groups of smaller rank.
Note that ``$\ind$'' denotes the compact induction, and ``$\equiv$'' denotes isomorphism of representations, up to semi-simplification.
\end{prop}

\begin{rmk}
Denote
$$V_{\ell,s}^{\pm}=\mathrm{Span}_k\{e_{\pm(\ell+1)},\cdots, e_{\pm(\ell+s)}\}\subset W_{\ell}.$$
When $w_0\in W_{\ell+s}$ or $H(k)$ is split, $Q_s'$ is the maximal parabolic subgroup of $L_{\ell,w_0}$ which preserves the isotropic subspace $\omega_b^t V_{\ell,s}^+\cap w_0^{\perp}$.
Otherwise, it is a proper subgroup of the parabolic subgroup. Moreover, $Q_{s,t}'$ is a subgroup of the maximal parabolic subgroup of $L_{\ell,\alpha}$, which preserves the isotropic subspace
$\eta_{s,t}^{-1} V_{\ell,s}^+\cap y_\alpha^{\perp}$.
\end{rmk}

With above preparation, we turn to the calculation of the twisted Jacquet module for any unramified local component of the residual
representation $\CE_{\tau\otimes\sigma}$.

Let $\tau$ be an irreducible unramified  representation of $\GL_{2n}(k)$ with central character $\omega_\tau=\mathbf{1}$.
Since $\tau$ is generic and self-dual, we may write $\tau$ as the full induced representation from the Borel subgroup as follows:
$$
\tau=\mu_1\times \cdots \times \mu_n\times \mu_n^{-1}\times \cdots \times \mu_1^{-1},
$$
where $\mu_i$'s are unramified characters on $k^\times$.
Also, let $\sigma$ be an irreducible unramified representation of $\SO(W_{2n})$. When $H=\SO(V)$ is $k$-split, the representation $\sigma$
is given by a unramified character $\mu_0$ of $k^\times$; and when $H$ is $k$-quasisplit, $\SO(W_{2n})$ is a compact torus, and hence $\sigma$ is
the trivial representation.

Let $\pi_{\tau\otimes \sigma}$ be the unramified constituent of $\Ind_{Q_{2n}}^H \tau\cdot |\det|^{\frac{1}{2}}\otimes \sigma$.
We will calculate the Jacquet module $J_{\psi_{\ell,\alpha}}(\pi_{\tau\otimes \sigma})$, using Proposition \ref{Jacquet module 1}, and lead to
the following proposition, concerning the vanishing of $J_{\psi_{\ell,\alpha}}(\pi_{\tau\otimes \sigma})$.

\begin{prop}\label{local theory}
Let $\tau$ be an irreducible unramified  representation of $\GL_{2n}(k)$ with central character $\omega_\tau=\mathbf{1}$, and $\sigma$ an irreducible unramified
representation of $\SO(W_{2n})$. The following hold.
\begin{itemize}
 \item[(i)] Assume that $J_\delta$ is non-split over $k$.
 \begin{enumerate}
 \item If $w_0=y_\alpha$ for $\alpha\in k^\times$ such that $J_{\delta,\alpha}$ is split, then
 $J_{\psi_{\ell,\alpha}}(\pi_{\tau\otimes \sigma})=0$ for $\ell\geq n+1$.
 \item If $w_0\in V_0$, then $J_{\psi_{{2n},w_0}}(\pi_{\tau\otimes \sigma})=0$.
\end{enumerate}
\item[(ii)] Assume that $J_\delta$ splits over $k$.  For any choice of $\alpha\in k^\times$, $J_{\psi_{\ell,\alpha}}(\pi_{\tau\otimes \sigma})=0$ for $\ell\geq  n+1$.
\end{itemize}
\end{prop}

\begin{proof}
We suppose first that $H$ is non-split, i.e. $J_\delta$ is non-split over $k$.

In this case, $\pi_{\tau\otimes \sigma}$ is the unramified constituent of the representation of $H$ induced from the character (here $\sigma=\mathbf{1}$)
 \begin{equation}
 \mu_1|\cdot|^{\frac{1}{2}}\otimes\cdots \otimes\mu_n|\cdot|^{\frac{1}{2}}\otimes\mu_n^{-1}|\cdot|^{\frac{1}{2}}\otimes \cdots \otimes \mu_1^{-1}|\cdot|^{\frac{1}{2}}\otimes\mathbf{1}.
 \end{equation}
 Moreover, one can find a Weyl element of $\SO(V)$ which conjugates above character to
 $$
 \mu_1|\cdot|^{\frac{1}{2}}\otimes\mu_1|\cdot|^{-\frac{1}{2}}\otimes \cdots \otimes \mu_n|\cdot|^{\frac{1}{2}}\otimes\mu_n|\cdot|^{-\frac{1}{2}}\otimes \mathbf{1}.
 $$
 Then, induction by stages, one sees that $\pi_{\tau\otimes \sigma}$ is the unramified constituent of $\Ind_{Q_{2n}}^H \tau'\otimes \sigma$, where
$$\tau'=\Ind_{P_{2,\cdots, 2}}^{\GL_{2n}(k)}\mu_1(\det\!_{\GL_2})\otimes\cdots \otimes \mu_n(\det\!_{\GL_2}).$$
Then since the twisted Jacquet functor corresponding to the descent is exact, we get (i) (1)
by Proposition \ref{Jacquet module 1} (2) if we take $w_0=y_\alpha$ as in (\ref{local anisotropic vector}) (here $j=\wt m=2n$, and $\delta_\alpha=1$ by our assumption on $\alpha\in k^\times$).
Note that we have used the fact that $\tau'_{(n)}=0$.

If $w_0\in V_0$, then by the arguments above we see that we also need to consider the Jacquet module  $J_{\psi_{{2n},w_0}}(\Ind_{Q_{2n}}^{H}\tau'\otimes \sigma)$.
By Proposition \ref{Jacquet module 1} (3), we have (see the notation in Proposition \ref{Jacquet module 1})
$$J_{\psi_{{2n},w_0}}(\Ind_{Q_{2n}}^{H}\tau'\otimes \sigma)\equiv d_{\tau'}\cdot J_{\psi'_{{0},w_0}}(\sigma^{\omega_b^{2n}}).$$
Note that $d_{\tau'}$ is the dimension of the space of $\psi$-Whittaker functionals on $\tau'$, hence is zero by the construction of $\tau'$.
Then the above Jacquet module is zero, and this finishes part (i) of the proposition.


We turn to consider the case that $H$ is split with Witt index $2n+1$ over $k$. By our assumption, $\pi_{\tau\otimes \sigma}$ is the unramified constituent of a representation induced from the character
$$ \mu_1 | \cdot |^{\frac{1}{2}}\otimes\cdots \otimes\mu_n|\cdot|^{\frac{1}{2}}\otimes\mu_n^{-1}|\cdot|^{\frac{1}{2}}\otimes\cdots\otimes \mu_1^{-1}|\cdot|^{\frac{1}{2}}\otimes\mu_0.$$
If $n$ is even, $\pi_{\tau\otimes \sigma}$ is the unramified constituent of $\Ind_{Q_{2n}}^H \tau'\otimes \sigma$, where
$$\tau'=\Ind_{P_{2,\cdots, 2}}^{\GL_{2n}(k)}\mu_1(\det\!_{\GL_2})\otimes\cdots \otimes \mu_n(\det\!_{\GL_2}).$$
And if $n$ is odd, $\pi_{\tau\otimes \sigma}$ is the unramified constituent of $\Ind_{Q_{2n}}^H \tau'\otimes \sigma^{\omega_b^0}$ with
$\displaystyle{\omega_b^0=\begin{pmatrix}
             0&1\\1&0
            \end{pmatrix}}
$.
Thus, applying Proposition \ref{Jacquet module 1} (1) with $\wt m=2n+1$ and $j=2n$, one sees that $J_{\psi_{\ell,\alpha}}(\pi_{\tau\otimes \sigma})=0$ for $\ell\geq n+1$. Note that in this case we
have that $\delta_\alpha=0$ for any choice of $\alpha\in k^\times$. Note that we also use the fact ${\tau'}^{(\ell)}=0$ for $\ell>n$ and $\tau'_{(\ell)}=0$ for $\ell>n-1$.
\end{proof}


For later use, we write out the Jacquet module $J_{\psi_{\ell,\alpha}}(\pi_{\tau\otimes \sigma})$ for the above unramified $\tau$ and $\sigma$,
in the case that $\ell=n$ and the form $J_{\delta,\alpha}$ is split. Note that the
last condition means that the quadratic form $\delta_1 x^2+\delta_2 y^2+\alpha z^2$ represents $0$ over $k$.

\begin{prop}\label{Jacquet module: list}
Assume that $\tau$ and $\sigma$ are unramified representations as above, and the form $J_{\delta,\alpha}$ is split.  Then the Jacquet module can be
realized as follows:
$$
J_{\psi_{n,\alpha}}(\pi_{\tau\otimes \sigma})\prec \Ind_{B_{\SO(2n+1)}}^{\SO(2n+1)}\mu_1\otimes \cdots \otimes \mu_n.
$$
Here ``$\pi_1\prec\pi_2$'' denotes that $\pi_1$ is a subquotient of $\pi_2$.
\end{prop}

\begin{proof}
Suppose first that $J_\delta$ is non-split. We have that $\mathrm{Witt}(V)=2n$, $\mathrm{Witt}(k\cdot y_{-\alpha}+V_0)=1$
($J_{\delta,\alpha}$ is split),
and $\sigma=\mathbf{1}$. Conjugating by a Weyl element as in the proof of Proposition \ref{local theory}, one sees that
$\pi_{\tau\otimes \sigma}$ is the unramified constituent of $\Ind_{Q_{2n-2}}^H \tau'\otimes \sigma$, where
$$
\tau'=\Ind_{P_{2,\cdots, 2}}^{\GL_{2n}(k)}\mu_1(\det\!_{\GL_2})\otimes\cdots \otimes \mu_{n}(\det\!_{\GL_2}).
$$
Now applying Proposition \ref{Jacquet module 1} (2) for $\ell=n$, $j=2n=\wt m$, and using the fact that ${\tau'}_{(n)}=0$, we have
$$
J_{\psi_{n,\alpha}}(\pi_{\tau\otimes \sigma})\prec  ind_{Q_\alpha'}^{L_{n,\alpha}}|\det(\cdot )|^{\frac{1-n}{2}}{\tau'}^{(n)}\otimes J_{\psi'_{0,v_{\alpha}}}({\sigma}^{\omega_b^n}).
$$

For the case that $J_\delta$ splits, we note that $\wt m=\mathrm{Witt}(V)=2n+1$, and $\mathrm{Witt}(k\cdot y_{-\alpha}+V_0)=0$.
Applying  Proposition \ref{Jacquet module 1} (1) for $j=2n$ and $\ell=n$, we see that
$$
J_{\psi_{n,\alpha}}(\pi_{\tau\otimes \sigma})\prec ind_{Q_{n}'}^{L_{n,\alpha}}|\det(\cdot )|^{\frac{1-n}{2}}{\tau'}^{(n)}\otimes J_{\psi'_{0,\alpha}}(\sigma^{\omega_b^n}).
$$
Note that $\displaystyle{|\det(\cdot)|^{\frac{1-n}{2}}{\tau'}^{(n)}=\Ind_{B_{\GL_{n}}}^{\GL_{n}}\mu_1\otimes \cdots \otimes \mu_{n}}$, and
the two Jacquet modules $J_{\psi'_{0,v_{\alpha}}}({\sigma}^{\omega_b^n})$ and $J_{\psi'_{0,\alpha}}(\sigma^{\omega_b^n})$ are
just the restriction of $\sigma^{\omega_b^n}$ of $\SO^\delta_2(k)$ to the trivial group. Therefore, the result follows.
\end{proof}

\subsection{Cuspidality of the twisted automorphic descent}

Now we go back to the global case and show that the twisted automorphic descent $\sigma_{\psi_{n,\alpha}}(\CE_{\tau\otimes\sigma})$
is cuspidal (but it may be zero). This result is a combination of the tower property of descent (\cite[Chapter 5]{GRS11})
and the local result of the previous subsection.

\begin{prop}\label{cuspidality}
For all $\ell>n$, the $\psi_{\ell,w_0}$-Fourier coefficients of the
residual representation $\CE_{\tau\otimes\sigma}$ are zero for all anisotropic vector $w_0\in W_\ell$. In particular, the
twisted automorphic descent $\sigma_{\psi_{n,\alpha}}(\CE_{\tau\otimes\sigma})$ is cuspidal.
\end{prop}

\begin{proof}
We will show this proposition using the vanishing properties of the corresponding twisted Jacquet modules we have studied in the previous subsection. By Remark \ref{rmk on anisotropic vector},
we see that up to conjugation, we only need to consider the $\psi_{\ell,w_0}$-coefficients for $w_0=y_\alpha$ ($\alpha\in F^\times$) or $w_0\in V_0$.
Let $v$ be a finite local place of the number field $F$ at which all data involved in the $\psi_{\ell,w_0}$-Fourier coefficients of $\CE_{\tau\otimes\sigma}$ is unramified. The unramified local component at $v$ of the residual representation
$\CE_{\tau\otimes\sigma}$ is denoted by $\CI_{\tau_v\otimes\sigma_v}$. For any integer $\ell>n$, if $\CE_{\tau\otimes\sigma}$ has nonzero $\psi_{\ell,w_0}$-Fourier coefficients, then the corresponding local Jacquet module
$J_{\psi_{\ell,w_0}}(\CI_{\tau_v\otimes\sigma_v})$ is nonzero. But the latter is zero according to Proposition \ref{local theory}.

Now we use the tower property as in \cite{GRS11} to prove the cuspidality of the twisted automorphic descent
$\sigma_{\psi_{n,\alpha}}(\CE_{\tau\otimes\sigma})$ (here $\ell=n$, and $w_0=y_{\alpha}$ by our choice). The idea is to calculate all constant terms along the unipotent radical of maximal
parabolic subgroups of $\SO^{\delta,\alpha}_{2n+1}$. Since the argument is inductive, we consider more general situation as the book
\cite{GRS11} did for the tower property.

 We denote the Witt index of the space $W_\ell\cap y_\alpha^{\perp}$ by $\wt m_{\ell,\alpha}$.
 For $1\leq p \leq \wt m_{\ell,\alpha}$, let $Q_p^*$ be the standard maximal parabolic
 subgroup of $L_{\ell,\alpha}$, which preserves the totally isotropic subspace
 $V_{\ell,p}^+=\Span\{e_{\ell+1},\cdots, e_{\ell+p}\}\cap y_\alpha^{\perp}.$
 Denote by $U_p^*$ the unipotent radical of $Q_p^*$. For the $\psi_{\ell,\alpha}$-Fourier coefficients
 $f^{\psi_{\ell,\alpha}}$ for $f\in \CE_{\tau\otimes\sigma}$, we denote by
 $$
 c_p(f^{\psi_{\ell,\alpha}})=\int_{U_p^*(F)\bs U_p^*(\BA)}f^{\psi_{\ell,\alpha}}(u)\ \mathrm{d} u
 $$
 its constant term along $U_p^*$. By \cite[Theorem 7.3]{GRS11}, the constant term
 $c_p(f^{\psi_{n,\alpha}})$ is expressed as a sum of integrals of $f^{\psi_{{n+p},\alpha}}$ if we have
 \begin{equation}\label{condition-cuspidality}
 (f^{U_{p-i}})^{\psi_{{n+i},\alpha}}=0. \qquad (0\leq i\leq p-1)
 \end{equation}
Recall that $U_{j}$ is the unipotent subgroup of the parabolic subgroup $Q_j$ of $H$, which preserves the isomorphic subspace
$V_j^+=\Span\{e_1,\cdots, e_j\}$ (see \S 2.1).

We claim that the condition (\ref{condition-cuspidality}) holds. In fact, by the cuspidal support of the Eisenstein series
$E(h,s,\phi_{\tau\otimes\sigma})$, which produces the residual representation $\CE_{\tau\otimes\sigma}$, the constant term
$E^{U_j}(h,s,\phi_\pi)$ is always zero for any $1\leq j\leq p$ and $1\leq p\leq n$. This implies the condition (\ref{condition-cuspidality}).

It follows that the constant term $c_p(f^{\psi_{n,\alpha}})$ is a sum of integrals of $f^{\psi_{{n+p},\alpha}}$, which are always zero by
the discussion at the beginning of this proof. Therefore, all the constant terms $c_p(f^{\psi_{n,\alpha}})$ are zero, which implies that
the twisted automorphic descent $\sigma_{\psi_{n,\alpha}}(\CE_{\tau\otimes\sigma})$ is cuspidal.

\end{proof}


\section{Certain Fourier Coefficients of the Residual Representation}\label{sec:FC}


We prove Parts (3), (5) and (6) of Theorem \ref{Main theorem of descent} in this section. Before that, we need to establish
certain properties of the Fourier coefficients of the residual representation $\CE_{\tau \otimes \sigma}$.
We are going to adapt the notation and the arguments used in \cite{JL15} in the proofs given in this section. In particular,
Lemma 2.5 of \cite{JL15} (see also \cite[Corollary 7.1]{GRS11}) is a technical key, which will be used many times in the proofs.


In this section, we will let $E_{i,j}$ be the unipotent subgroup of $\SO^{\delta}_{4n+2}$ consists of elements $u$ with 1's on the diagonal, and for $k\neq \ell$, $u_{k,\ell}=0$ unless $(k,\ell)=(i,j)$ or $(k,\ell)=(4n+3-j,4n+3-i)$. For $x \in F$, let $E_{i,j}(x)$ be the element $u \in E_{i,j}$ with $u_{i,j}=x$. 

\subsection{Fourier coefficients associated to the partition $[(2n)^21^2]$}
We first show that the residual representation $\CE_{\tau \otimes \sigma}$ has a nonzero Fourier coefficient associated to the
partition $[(2n)^21^2]$, based on the cuspidal support of $\CE_{\tau \otimes \sigma}$,
which serves a base for Parts (3), (5) and (6) of Theorem \ref{Main theorem of descent}.

\begin{prop}\label{thm1}
The residual representation $\CE_{\tau \otimes \sigma}$ has a nonzero Fourier coefficient attached to the partition $[(2n)^21^2]$.
\end{prop}

\begin{proof}
By \cite[I.6]{W01}, there is only one $F$-rational nilpotent orbit in each $F$-stable nilpotent orbit in the Lie algebra
of $\SO^{\delta}_{4n+2}(F)$
attached to the partition $[(2n)^21^2]$. Moreover, a nilpotent element in the only $F$-rational nilpotent orbit can
be chosen as follows.

Let $X_{[(2n)^21^2]} = \sum_{i=1}^{n-1} E_{i+1,i}(\frac{1}{2})$. Then $X_{[(2n)^21^2]}$ is a
representative of the $F$-rational orbit corresponding to the partition $[(2n)^21^2]$. To define the Fourier coefficients attached to
the partition $[(2n)^21^2]$, we introduce a one-dimensional toric subgroup $\CH_{[(2n)^21^2]}$ as follows: for $t \in F^{\times}$,
\begin{equation}\label{equtorus}
\CH_{[(2n)^21^2]}(t):=\diag(t^{2n-1}, t^{2n-3}, \ldots, t^{1-2n}, 1, 1, t^{2n-1}, t^{2n-3}, \ldots, t^{1-2n}).
\end{equation}
It is easy to see that under the adjoint action,
$$\Ad(\CH_{[(2n)^21^2]}(t))(X_{[(2n)^21^2]}) = t^{-2} X_{[(2n)^21^2]}, \forall t \in F^{\times}.$$

Let $\mathfrak{g}$ be the Lie algebra of $\SO^{\delta}_{4n+2}(F)$. Under the
adjoint action of $\mathcal{H}_{[(2n)^21^2]}$,
$\mathfrak{g}$ has the following direct sum decomposition into
$\mathcal{H}_{[(2n)^21^2]}$-eigenspaces:
\begin{equation}\label{ssd}
\mathfrak{g} = \mathfrak{g}_{-m} \oplus \cdots \oplus \mathfrak{g}_{-2} \oplus \mathfrak{g}_{-1} \oplus \mathfrak{g}_{0} \oplus \mathfrak{g}_{1} \oplus \mathfrak{g}_{2} \oplus \cdots \oplus \mathfrak{g}_{m},
\end{equation}
for some positive integer $m$, where
$\mathfrak{g}_{l}:= \{X \in \mathfrak{g}\ |\ \Ad(\mathcal{H}_{[(2n)^21^2]}(t))(X) = t^{l}X\}$.
Let $V_{[(2n)^21^2], j}$ ($j=1, \ldots, m$) be the unipotent
subgroup of $\SO^{\delta}_{4n+2}(F)$ with Lie algebra $\oplus_{l=j}^m \mathfrak{g}_l$.
Let $L_{[(2n)^21^2]}$ be the algebraic subgroup of $\SO^{\delta}_{4n+2}(F)$
with Lie algebra $\mathfrak{g}_{0}$.
Then we define a character of $V_{[(2n)^21^2], 2}$ as follows:
for $v \in V_{[(2n)^21^2], 2}(\BA)$ and for a nontrivial character $\psi$ of $F\bks\BA$,
\begin{align}\label{ch}
\begin{split}
\psi_{[(2n)^21^2]}(v) := & \ \psi(\tr(X_{[(2n)^21^2]}\log(v)))\\
= & \ \psi(v_{1,2}+v_{2,3}+\cdots+v_{2n-1, 2n}).
\end{split}
\end{align}

For an arbitrary automorphic form $\varphi$ on $\SO^{\delta}_{4n+2}(\BA)$, the $\psi_{[(2n)^21^2]}$-Fourier coefficient
of $\varphi$ is defined by
\begin{equation}\label{fc}
\varphi^{\psi_{[(2n)^21^2]}}(g):=\int_{[V_{[(2n)^21^2], 2}]}
\varphi(vg) \psi^{-1}_{[(2n)^21^2]}(v) \mathrm{d}v,
\end{equation}
where for a group $G$, we denote the quotient $G(F) \bs G(\BA)$ simply by $[G]$.
When an irreducible automorphic representation $\pi$ of $\SO^{\delta}_{4n+2}(\BA)$ is generated by automorphic forms $\varphi$,
we say that $\pi$ has a nonzero $\psi_{[(2n)^21^2]}$-Fourier coefficient or
a nonzero Fourier coefficient attached to $[(2n)^21^2]$ if there exists an
automorphic form $\varphi$ in the space of $\pi$ with a nonzero $\psi_{[(2n)^21^2]}$-Fourier coefficient
$\varphi^{\psi_{[(2n)^21^2]}}(g)$.

Note that by \cite[Section 1.7]{MW87}, $V_{[(2n)^21^2], 1} / \ker_{V_{[(2n)^21^2], 2}}(\psi_{[(2n)^21^2]})$ is a Heisenberg group $V_{[(2n)^21^2], 1} / V_{[(2n)^21^2], 2} \oplus Z$, with center
$$Z = V_{[(2n)^21^2], 2} / \ker_{V_{[(2n)^21^2], 2}}(\psi_{[(2n)^21^2]}).$$
And, one can see that $(E_{2n+1,n+1} E_{n,2n+2}) \oplus (E_{n,2n+1} E_{2n+2,n+1})$ is a complete polarization of $$V_{[(2n)^21^2], 1} / V_{[(2n)^21^2], 2}.$$
Let $Y = E_{2n+1,n+1} E_{n,2n+2}$.

To show that the residual representation $\CE_{\tau \otimes \sigma}$ has a nonzero Fourier coefficient attached to the partition $[(2n)^21^2]$,
we need to show that there is a $\varphi \in \CE_{\tau \otimes \sigma}$, such that $\varphi^{\psi_{[(2n)^21^2]}}$ is non-vanishing. By a similar argument as in
\cite[Lemma 1.1]{GRS03}, $\varphi^{\psi_{[(2n)^21^2]}}$ is non-vanishing if and only if the following integral is non-vanishing:
\begin{equation}\label{addy}
\int_{[V_{[(2n)^21^2], 2}Y]}
\varphi(vyg) \psi^{-1}_{[(2n)^21^2]}(v)\  \mathrm{d}v\mathrm{d}y.
\end{equation}

First, one can see that the following quadruple satisfies all the conditions for \cite[Corollary 7.1]{GRS11} (which is still true for the case of $\SO_{4n+2}^{\delta}(\BA)$):
\begin{equation}\label{changey}
(V_{[(2n)^21^2], 2}X_{E_{n,2n+2}}, \psi_{[(2n)^21^2]}, E_{2n+1,n+1}, E_{n,2n+1}).
\end{equation}
Applying \cite[Corollary 7.1]{GRS11}, the integral in \eqref{addy} is non-vanishing if and only the following integral is non-vanishing
\begin{equation}\label{addy3}
\int_{[V_{[(2n)^21^2], 2}Y']}
\varphi(vyg) \psi^{-1}_{[(2n)^21^2]}(v) \ \mathrm{d}v\mathrm{d}y,
\end{equation}
with $Y' = E_{n,2n+1} E_{n,2n+2}$.

Let $W=V_{[(2n)^21^2], 2}Y'$, the elements in which have the following form:
\begin{equation}\label{thm1equ12}
w=\begin{pmatrix}
z & q_1 & q_2\\
0 & I_2 & q_1^*\\
0 & 0 & z^*
\end{pmatrix}\begin{pmatrix}
I_{2n} & 0 & 0\\
p_1& I_2 & 0\\
p_2 &p_1^*& I_{2n}
\end{pmatrix} \in \SO^{\delta}_{4n+2},
\end{equation}
where $z \in Z_{2n}$, the standard maximal unipotent subgroup of $\GL_{2n}$; $q_1 \in \rm{Mat}_{2n \times 2}$, with $q_1(i,j)=0$,
for $n+1 \leq i \leq 2n$ and $1 \leq j \leq 2$;
$q_2 \in \Mat_{2n \times 2n}$, with
$q_2(i,j)=0$, for $i \geq j$; $p_1 \in \rm{Mat}_{2 \times 2n}$, with $p_1(i,j)=0$ for $1 \leq i \leq 2$ and $1 \leq j \leq n+1$;
$p_2 \in \Mat_{2n \times 2n}$ (with certain additional property we do not specify here), with $p_2(i,j)=0$, for $i \geq j$;
all entries marked with a * are determined by the other entries and symmetry. Let $\psi_W(w) := \psi_{[(2n)^21^2]}(v)$,
for $w=vy \in W$, where $v \in V_{[(2n)^21^2], 2}$ and $y \in Y'$. For $w\in W$ of form in $\eqref{thm1equ12}$,
$\psi_W(w)=\psi(\sum_{i=1}^{2n-1} z_{i,i+1})$.

The idea of proving the non-vanishing property of the integral in \eqref{addy3} is to express the Fourier coefficient in \eqref{addy3} in terms of the coefficient in \eqref{thm1equ10} (which is non-vanishing) and the coefficients in Lemma \ref{vanishinglem} (which are identically zero) using exchange of unipotent subgroups and Fourier expansion. More explicitly, we use the lower triangular part of elements $w$ in \eqref{thm1equ12} to fill the zeros in the upper triangular part. We proceed one row/column at a time, from top to bottom in the upper triangular part, and from left to right in the lower triangular part. For the first $n-1$ rows, the number of zeros in the upper triangular part is exactly equal to the number of nonzero entries in the corresponding column in the lower triangular part. These rows will be treated via exchange of unipotent subgroups, using \cite[Corollary 7.1]{GRS11}.
For the $i=n, \ldots, 2n-1$ rows, the number of zeros in the $i$-th row in the upper triangular part exceeds the number of nonzero entries in the corresponding column in the lower triangular part by one. For the $n$-th row, we first add the constant integral over the one-dimensional subgroup $E_{n,3n+2}$, then do exchange of unipotent subgroups. For the row $i=n+1, \ldots, 2n-1$, we treat each row/column by a two-step process, that is, before exchanging unipotent subgroups for each row/column, one needs to take Fourier expansion along a one-dimensional subgroup $E_{i,4n+2-i}$, producing Fourier coefficients in Lemma \ref{vanishinglem}.

To continue, we define a sequence of unipotent subgroups as follows. Note that we have made a convention on unipotent subgroups in the second paragraph of Section 3.

For $1 \leq i \leq n-1$, $1 \leq j \leq i$, let $X^i_j=E_{i,2n+3+i-j}$ and $Y^i_j=E_{2n+3+i-j,i+1}$.
For $n \leq i \leq 2n-2$, $1 \leq j \leq 2n-i-1$, let $X^i_j=E_{i,4n+2-i-j}$ and $Y^i_j=E_{4n+2-i-j,i+1}$.
For $n+1 \leq i \leq 2n-1$, let $X_i = E_{i,2n+1} E_{i,2n+2}$ and $Y_i = E_{2n+1,i+1} E_{2n+2,i+1}$.

Let $\wt{W}$ be the subgroup of $W$ with elements of the form as in \eqref{thm1equ12}, but with the $p_1$ and $p_2$ parts zero. Let $\psi_{\wt{W}} = \psi_W|_{\wt{W}}$. For any subgroup of $W$ containing $\wt{W}$, we automatically extend $\psi_{\wt{W}}$ trivially to this subgroup and still denote the character by $\psi_{\wt{W}}$.

Next, we apply \cite[Corollary 7.1]{GRS11} (exchanging unipotent subgroups) to a sequence of quadruples. For $i$ going from 1 to $n-1$,
the following sequence of quadruples satisfy all the conditions for \cite[Corollary 7.1]{GRS11}:
\begin{align}\label{thm1equ20}
\begin{split}
& (\wt{W}_i \prod_{j=2}^i Y^i_j, \psi_{\wt{W}}, X^i_1, Y^i_1),\\
& (X^i_1 \wt{W}_i \prod_{j=3}^i Y^i_j, \psi_{\wt{W}}, X^i_2, Y^i_2),\\
& \cdots,\\
& (\prod_{j=1}^{k-1} X^i_j \wt{W}_i \prod_{j=k+1}^i Y^i_j, \psi_{\wt{W}}, X^i_k, Y^i_k),\\
& \cdots,\\
& (\prod_{j=1}^{i-1} X^i_j \wt{W}_i, \psi_{\wt{W}}, X^i_{i}, Y^i_{i}),
\end{split}
\end{align}
where
$$\wt{W}_i = \prod_{s=1}^{i-1} \prod_{j=1}^{s} X^s_j \wt{W} \prod_{t=n+1}^{2n-1} Y_t \prod_{j=1}^{2n-t-1} Y^t_j \prod_{k=i+1}^{n} \prod_{j=1}^{k} Y^k_j.$$
Applying \cite[Corollary 7.1]{GRS11} repeatedly to the above sequence of quadruples, one obtains that the integral in \eqref{addy3} is non-vanishing if and only if the following integral is non-vanishing:
\begin{equation}\label{thm1equ3}
\int_{[\wt{W}_{i}']} \varphi(wg) \psi^{-1}_{\wt{W}_{i}'}(w)\  \mathrm{d}w,
\end{equation}
where
\begin{equation}\label{thm1equ7}
\wt{W}_{i}' = \prod_{s=1}^{i} \prod_{j=1}^{s} X^s_j \wt{W} \prod_{t=n+1}^{2n-1} Y_t \prod_{j=1}^{2n-t-1} Y^t_j \prod_{k=i+1}^{n} \prod_{j=1}^{k} Y^k_j,
\end{equation}
and $\psi_{\wt{W}_{i}'}$ is extended from
$\psi_{\wt{W}}$ trivially.

To show the integral over $\wt{W}_{n-1}'$ in \eqref{thm1equ3} is non-vanishing, it suffices to show that the following integral is non-vanishing:
\begin{equation}\label{thm1equ13}
\int_{[E_{n,3n+2}]} \int_{[\wt{W}_{n-1}']} \varphi(wng) \psi^{-1}_{\wt{W}_{n-1}'}(w)\  \mathrm{d}w \mathrm{d}n,
\end{equation}
since it factors though the integral over $\wt{W}_{n-1}'$ in \eqref{thm1equ3}.
One can see that the following sequence of quadruples satisfy all the conditions for \cite[Corollary 7.1]{GRS11}:
\begin{align}\label{thm1equ14}
\begin{split}
& (E_{n,3n+2}\wt{W}_n \prod_{j=2}^{n-1} Y^n_j, \psi_{\wt{W}}, X^n_1, Y^n_1),\\
& (X^n_1 E_{n,3n+2}\wt{W}_n \prod_{j=3}^{n-1} Y^n_j, \psi_{\wt{W}}, X^n_2, Y^n_2),\\
& \cdots,\\
& (\prod_{j=1}^{k-1} X^n_j E_{n,3n+2}\wt{W}_n \prod_{j=k+1}^{n-1} Y^n_j, \psi_{\wt{W}}, X^n_k, Y^n_k),\\
& \cdots,\\
& (\prod_{j=1}^{n-2} X^n_j E_{n,3n+2}\wt{W}_n, \psi_{\wt{W}}, X^n_{n-1}, Y^n_{n-1}),
\end{split}
\end{align}
where
$$\wt{W}_n =
\prod_{t=1}^{n-1} \prod_{j=1}^{t} X^t_j \wt{W} \prod_{k=n+1}^{2n-1} Y_k \prod_{j=1}^{2n-k-1} Y^k_j.$$
Applying \cite[Corollary 7.1]{GRS11} repeatedly to the above sequence of quadruples, the integral in \eqref{thm1equ13} is non-vanishing if and only if the following integral is non-vanishing:
\begin{equation}\label{thm1equ15}
\int_{[\wt{W}_{n}']} \varphi(wg) \psi^{-1}_{\wt{W}_{n}'}(w) \ \mathrm{d}w,
\end{equation}
where
\begin{equation}\label{thm1equ16}
\wt{W}_{n}' = E_{n,3n+2} \prod_{j=1}^{n-1} X^n_j
\prod_{t=1}^{n-1} \prod_{j=1}^{t} X^t_j \wt{W} \prod_{k=n+1}^{2n-1} Y_k \prod_{j=1}^{2n-k-1} Y^k_j,
\end{equation}
and $\psi_{\wt{W}_{n}'}$ is extended from
$\psi_{\wt{W}}$ trivially.

For $i$ going from $n+1$ to $2n-1$, define
\begin{equation}\label{thm1equ17}
\wt{W}_{i}' = \prod_{s=n+1}^{i} X_s \prod_{j=1}^{2n-s-1} X^s_j \prod_{\ell=n}^{i} E_{\ell, 4n+3-\ell}
\prod_{t=1}^{n-1} \prod_{j=1}^{t} X^t_j \wt{W} \prod_{k=i+1}^{2n-1} Y_k \prod_{j=1}^{2n-k-1} Y^k_j,
\end{equation}
and $\psi_{\wt{W}_{i}'}$ is extended from
$\psi_{\wt{W}}$ trivially.
We claim that the following integral is non-vanishing:
\begin{equation}\label{thm1equ18}
\int_{[\wt{W}_{i-1}']} \varphi(wg) \psi^{-1}_{\wt{W}_{i-1}'}(w) \ \mathrm{d}w,
\end{equation}
if and only if the following integral is non-vanishing:
\begin{equation}\label{thm1equ19}
\int_{[\wt{W}_{i}']} \varphi(wg) \psi^{-1}_{\wt{W}_{i}'}(w) \ \mathrm{d}w,
\end{equation}

Indeed, first, we take the Fourier expansion of the integral in \eqref{thm1equ18} along $E_{i,4n+2-i}$. Under the action of $\GL_1$, we get two kinds of Fourier coefficients corresponding
to the two orbits of the dual of $[E_{i,4n+2-i}]$: the trivial one and the non-trivial one. By Lemma \ref{vanishinglem}, which will be stated and proved right after the proof of the theorem,
all the Fourier coefficients corresponding to the non-trivial orbit are identically zero. Therefore only the Fourier coefficient attached to the trivial orbit survives. When $i=2n-1$, the Fourier coefficient attached to the trivial orbit is exactly the integral in \eqref{thm1equ19}, that is, the claim is proved. When $n+1 \leq i \leq 2n-2$, one can see that the following sequence of quadruples satisfy all the conditions for \cite[Corollary 7.1]{GRS11}:
\begin{align*}
& (E_{i,4n+2-i}\wt{W}_i Y_i \prod_{j=2}^{2n-i-1} Y^i_j, \psi_{\wt{W}}, X^i_1, Y^i_1),\\
& (X^i_1 E_{i,4n+2-i}\wt{W}_i Y_i \prod_{j=3}^{2n-i-1} Y^i_j, \psi_{\wt{W}}, X^i_2, Y^i_2),\\
& \cdots,\\
& (\prod_{j=1}^{k-1} X^i_jE_{i,4n+2-i}\wt{W}_i Y_i  \prod_{j=k+1}^{2n-i-1} Y^i_j, \psi_{\wt{W}}, X^i_k, Y^i_k),\\
& \cdots,\\
& (\prod_{j=1}^{2n-i-2} X^i_j E_{i,4n+2-i}\wt{W}_i Y_i , \psi_{\wt{W}}, X^i_{2n-i-1}, Y^i_{2n-i-1}),\\
& (\prod_{j=1}^{2n-i-1} X^i_j E_{i,4n+2-i} \wt{W}_i, \psi_{\wt{W}}, X_i, Y_i),
\end{align*}
where
$$\wt{W}_i =\prod_{s=n+1}^{i-1} X_s \prod_{j=1}^{2n-s-1} X^s_j \prod_{\ell=n}^{i-1} E_{\ell,4n+2-\ell}
\prod_{t=1}^{n-1} \prod_{j=1}^{t} X^t_j \wt{W} \prod_{k=i+1}^{2n-1} Y_k \prod_{j=1}^{2n-k-1} Y^k_j.$$
Applying \cite[Corollary 7.1]{GRS11} repeatedly to the above sequence of quadruples, we deduce that the Fourier coefficient attached to the trivial orbit above is non-vanishing if and only if the integral in \eqref{thm1equ19} is non-vanishing. This proves the claim.

It is easy to see that elements of $\wt{W}_{2n-1}'$ have the following form:
\begin{equation}\label{thm1equ6}
w=\begin{pmatrix}
z & q_1 & q_2\\
0 & I_2 & q_1^*\\
0 & 0 & z^*
\end{pmatrix},
\end{equation}
where $z \in Z_{2n}$, the standard maximal unipotent subgroup of $\GL_{2n}$; $q_1 \in \rm{Mat}_{2n \times 2}$, with $q_1(2n,j)=0$, for $1 \leq j \leq 2$;
$q_2 \in \Mat_{2n \times 2n}$ with certain symmetry property we do not specify here; all entries marked with a $*$ are determined by the other entries and symmetry. For $w\in {\wt{W}_{2n-1}'}$ of form in $\eqref{thm1equ6}$,
$\psi_{\wt{W}_{2n-1}'}(w)=\psi(\sum_{i=1}^{2n-1} z_{i,i+1})$.

Now, we need to take the Fourier expansion of the integral over $\wt{W}_{2n-1}'$ in \eqref{thm1equ19} along
$$E_{2n,2n+1}E_{2n,2n+2},$$
which can be identified with a certain quadratic extension $E$ of $F$, and a torus which can be identified with $F^{\times} \times (E^{\times})^1$ acting on it. Here $(E^{\times})^1 \subset E^{\times}$ is the kernel of the norm map to $F^{\times}$.
This Fourier expansion gives us two kinds of Fourier coefficients corresponding to the two orbits of characters: the trivial one and the non-trivial one, and the Fourier coefficients corresponding to the non-trivial orbit are generic Fourier coefficients.
Since $\CE_{\tau \otimes \sigma}$ is not generic, only the Fourier coefficient corresponding to the trivial orbit survives. Therefore,
the integral over $\wt{W}_{2n-1}'$ in \eqref{thm1equ19} becomes
\begin{align}\label{thm1equ10}
\begin{split}
& \int_{[E_{2n,2n+1}E_{2n,2n+2}]} \int_{[\wt{W}_{2n-1}']} \varphi(wxg) \psi^{-1}_{\wt{W}_{2n-1}'}(w)\  \mathrm{d}w\mathrm{d}x\\
= & \ \int_{N_{2n}} \varphi(ng) \psi^{-1}_{N_{2n}}(n) \ \mathrm{d}n,
\end{split}
\end{align}
where $N_{2n}$ is the unipotent radical of the parabolic subgroup with Levi isomorphic to $\GL_1^{2n} \times \SO_2^{\delta}$,
and $\psi_{N_{2n}}(n)=\psi(\sum_{i=1}^{2n-1} n_{i,i+1})$.

Let $U$ be the unipotent radical of the parabolic subgroup $P=MU$ with Levi $M \cong \GL_{2n} \times \SO_2^{\delta}$. Write $N_{2n} = U N_{2n}'$ with $N_{2n}'=M \cap N_{2n}$.
Then the integral in \eqref{thm1equ10} becomes
\begin{equation}\label{thm1equ11}
\int_{[N_{2n}']} \varphi^U(ng) \psi^{-1}_{N_{2n}'}(n)\ \mathrm{d}n,
\end{equation}
where $\varphi^U$ is the constant term of $\varphi$ along $U$,
and $\psi_{N_{2n}'}=\psi_{N_{2n}}|_{N_{2n}'}$.
By a similar calculation as in \cite{JLZ13},
one can see that
$\varphi^U \in A(U(\BA)M(F)\bs \SO^{\delta}_{4n+2}(\BA))_{\lvert \cdot \rvert^{-\frac{1}{2}} \tau \otimes \sigma}$.
Since $\tau$ is generic, it follows that the integral in \eqref{thm1equ11} is non-vanishing.
Therefore, the integral in \eqref{addy}, hence the Fourier coefficient attached to $[(2n)^21^2]$ in \eqref{fc} is non-vanishing.

This completes the proof of the proposition.
\end{proof}

Now we prove the lemma used in the above proof.

\begin{lem}\label{vanishinglem}
For any $\varphi \in \CE_{\tau \otimes \sigma}$, the following integral is identically zero:
\begin{equation}\label{vanishinglemequ1}
\int_{[E_{i+1,4n+1-i}]}\int_{[\wt{W}_{i}']} \varphi(wxg) \psi^{-1}_{\wt{W}_{i}'}(w) \psi^{-1}(ax)\ \mathrm{d}w\mathrm{d}x,
\end{equation}
where $n \leq i \leq 2n-2$, $a \in F^{\times}$,
$\wt{W}_{n}'$ is as in \eqref{thm1equ16} with $i=n$ there,
and $\wt{W}_{i}'$ is as in \eqref{thm1equ17} for $n+1 \leq i \leq 2n-2$.
\end{lem}

\begin{proof}
We continue to use notation introduced in Proposition \ref{thm1}.
First note that elements in $\wt{W}_{i}'$ have the following form
\begin{equation}\label{vanishinglemequ2}
w=\begin{pmatrix}
z & q_1 & q_2\\
0 & I_2 & q_1^*\\
0 & 0 & z^*
\end{pmatrix}\begin{pmatrix}
I_{2n} & 0 & 0\\
p_1& I_2 & 0\\
p_2 &p_1^*& I_{2n}
\end{pmatrix} \in \SO^{\delta}_{4n+2},
\end{equation}
where $z \in Z_{2n}$, the standard maximal unipotent subgroup of $\GL_{2n}$; $q_1 \in \rm{Mat}_{2n \times 2}$, with $q_1(k,j)=0$, for $i + 1 \leq k \leq 2n$ and $1 \leq j \leq 2$;
$q_2 \in \Mat_{2n \times 2n}$, with
$q_2(k,j)=0$, for $i+1\leq k \leq 2n$ and $1 \leq j \leq 2n-i-1$; $p_1 \in \rm{Mat}_{2 \times 2n}$, with $p_1(k,j)=0$ for $1 \leq k \leq 2$ and $1 \leq j \leq i+1$; $p_2 \in \Mat_{2n \times 2n}$, with $p_2(k,j)=0$, for $2n-i \leq k \leq 2n$ or $1 \leq j \leq i+1$; all entries marked with a $*$ are determined by the other entries and symmetry. For $w\in \wt{W}_{i}'$ of form in $\eqref{vanishinglemequ2}$,
$\psi_{\wt{W}_{i}'}(w)=\psi(\sum_{i=1}^{2n-1} z_{i,i+1})$.

In the rest of this proof, we set $\wt{W}_{i}^{(1)}:=\wt{W}_{i}'$, and denote by $\wt{W}_{i}^{(2)}$ the subgroup of $\wt{W}_{i}^{(1)}$
consisting of the elements of the following form
\begin{equation*}
\begin{pmatrix}
I_{i+1} & 0 & 0 & 0 & 0\\
0 & z & 0 & 0 & 0\\
0 & 0 & I_2 & 0 & 0\\
0 & 0 & 0 & z^* & 0\\
0 & 0 & 0 & 0 & I_{i+1}
\end{pmatrix},
\end{equation*}
where $z \in Z_{2n-i-1}$, the standard maximal unipotent subgroup of $\GL_{2n-i-1}$.
Write
$$
\wt{W}_{i}^{(1)} = \wt{W}_{i}^{(2)}\wt{W}_{i}^{(3)},
$$
where $\wt{W}_{i}^{(3)}$ is a subgroup of $\wt{W}_{i}^{(1)}$ consisting elements with
$\wt{W}_{i}^{(2)}$-part being trivial.
Let $\psi_{\wt{W}_{i}^{(2)}} = \psi_{\wt{W}_{i}^{(1)}}|_{\wt{W}_{i}^{(2)}}$
and $\psi_{\wt{W}_{i}^{(3)}} = \psi_{\wt{W}_{i}^{(1)}}|_{\wt{W}_{i}^{(3)}}$.
Then the integral in \eqref{vanishinglemequ1} can be written as
\begin{equation}\label{vanishinglemequ3}
\int_{[\wt{W}_{i}^{(2)}]} \int_{[E_{i+1,4n+1-i}]}\int_{[\wt{W}_{i}^{(3)}]} \varphi(w_1 x w_2 g) \psi^{-1}_{\wt{W}_{i}^{(2)}}(w_2) \psi^{-1}(ax)\psi^{-1}_{\wt{W}_{i}^{(3)}}(w_1)\
\mathrm{d}w_1\mathrm{d}x\mathrm{d}w_2.
\end{equation}
Therefore, to show the integral in \eqref{vanishinglemequ1} is identically zero, it is suffices to show the following integral is identically zero
\begin{equation}\label{vanishinglemequ4}
\int_{[E_{i+1,4n+1-i}]}\int_{[\wt{W}_{i}^{(3)}]} \varphi(wx g) \psi^{-1}(ax)\psi^{-1}_{\wt{W}_{i}^{(3)}}(w)\ \mathrm{d}w\mathrm{d}x.
\end{equation}

To continue, we define a sequence of unipotent subgroups as follows. For $1 \leq j \leq 2n-i-2$, let $X^{i+1}_j=E_{i+1,4n+1-i-j}$ and $Y^{i+1}_j=E_{4n+1-i-j,i+1}$. Let
$X_{i+1} = E_{i+1,2n+1} E_{i+1,2n+2}$ and $Y_{i+1} = E_{2n+1,i+2}E_{2n+2,i+2}$.
Write $\wt{W}_{i}^{(3)} = \wt{W}_{i}^{(4)} \prod_{j=1}^{2n-i-2} Y^{i+1}_j Y_{i+1}$ with $\wt{W}_{i}^{(4)}$ is a subgroup of $\wt{W}_{i}^{(3)}$ consisting of elements with $\prod_{j=1}^{2n-i-2} Y^{i+1}_j Y_{i+1}$-part being trivial.
Let $\psi_{\wt{W}_{i}^{(4)}} = \psi_{\wt{W}_{i}^{(3)}}|_{\wt{W}_{i}^{(4)}}$.
Then, one can see that the following sequence of quadruples satisfy the conditions in \cite[Corollary 7.1]{GRS11}:
\begin{align*}
& (E_{i+1,4n+1-i} \wt{W}_{i}^{(4)} Y_{i+1} \prod_{j=2}^{2n-i-2} Y^{i+1}_j, \psi_{\wt{W}^{(4)}}, X^{i+1}_1, Y^{i+1}_1),\\
& (X^{i+1}_1 E_{i+1,4n+1-i}\wt{W}_{i}^{(4)} Y_{i+1} \prod_{j=3}^{2n-i-2} Y^{i+1}_j, \psi_{\wt{W}^{(4)}}, X^{i+1}_2, Y^{i+1}_2),\\
& \cdots,\\
& (\prod_{j=1}^{k-1} X^{i+1}_j E_{i+1,4n+1-i}\wt{W}_{i}^{(4)} Y_{i+1} \prod_{j=k+1}^{2n-i-2} Y^{i+1}_j, \psi_{\wt{W}^{(4)}}, X^{i+1}_k, Y^{i+1}_k),\\
& \cdots,\\
& (\prod_{j=1}^{2n-i-3} X^{i+1}_j E_{i+1,4n+1-i}\wt{W}_{i}^{(4)} Y_{i+1}, \psi_{\wt{W}^{(4)}}, X^{i+1}_{2n-i-2}, Y^{i+1}_{2n-i-2}),\\
& (\prod_{j=1}^{2n-i-2} X^{i+1}_j E_{i+1,4n+1-i}\wt{W}_{i}^{(4)}, \psi_{\wt{W}^{(4)}}, X_{i+1}, Y_{i+1}).
\end{align*}
After applying \cite[Corollary 7.1]{GRS11} to the above sequence of quadruples, one can see that the integral in \eqref{vanishinglemequ4}
is identically zero if and only if the following integral is identically zero:
\begin{equation}\label{vanishinglemequ5}
\int_{[\prod_{j=1}^{2n-i-2} X^{i+1}_j E_{i+1,4n+1-i}]}\int_{[\wt{W}_{i}^{(4)}]} \varphi(wx g) \psi^{-1}(ax)\psi^{-1}_{\wt{W}_{i}^{(4)}}(w)\ \mathrm{d}w\mathrm{d}x.
\end{equation}
Finally, the integral in \eqref{vanishinglemequ5} contains an inner integral which is exactly a Fourier coefficient attached to the partition $[(2i+3)1^{4n-2i-1}]$.
Since $n \leq i \leq 2n-2$, by Proposition \ref{cuspidality}, any Fourier coefficient attached to the partition $[(2i+3)1^{4n-2i-1}]$ will be identically zero. Therefore, the integral in
\eqref{vanishinglemequ5} is identically zero, and hence the integral in \eqref{vanishinglemequ1} is identically zero.
This completes the proof of the lemma.
\end{proof}

\subsection{Parts (3), (5) and (6) of Theorem \ref{Main theorem of descent}}
We are ready to prove Parts (3), (5) and (6) of Theorem \ref{Main theorem of descent} based on the preparation in Section 3.1.

\begin{prop}[Part (3) of Theorem \ref{Main theorem of descent}]\label{thm2}
The residual representation $\CE_{\tau \otimes \sigma}$ has a nonzero Fourier coefficient attached to the partition $[(2n+1)1^{2n+1}]$.
\end{prop}

\begin{proof}
We prove this theorem by contradiction. Assume that the residual representation $\CE_{\tau \otimes \sigma}$ has no nonzero Fourier coefficients attached to the partition $[(2n+1)1^{2n+1}]$. In the following, We will continue to use the notation introduced in Proposition \ref{thm1}.

In Proposition \ref{thm1}, we have proved that the integral in \eqref{addy} is non-vanishing. We name this integral as follows:
\begin{equation}\label{addy2}
f(g):=\int_{[V_{[(2n)^21^2], 2}Y]}
\varphi(vyg) \psi^{-1}_{[(2n)^21^2]}(v) \ \mathrm{d}v\mathrm{d}y, \quad  g \in \SO^{\delta}_{4n+2}(\BA), \ \varphi \in \CE_{\tau \otimes \sigma}.
\end{equation}

Let
\begin{equation}\label{weylelement}
\omega=\begin{pmatrix}
0 & 0 & I_{2n}\\
0 & I_2 & 0\\
I_{2n} & 0 & 0
\end{pmatrix}.
\end{equation}
Then it is easy to see that $f(g)=f(\omega g)$.

Applying the equality in \cite[Lemma 7.1]{GRS11} to the quadruple in \eqref{changey}, one can see that
\begin{equation}\label{addy4}
f(g)=\int_{E_{2n+1,n+1}(\BA)}\int_{[V_{[(2n)^21^2], 2}Y']}
\varphi(vyxg) \psi^{-1}_{[(2n)^21^2]}(v)\ \mathrm{d}v\mathrm{d}y\mathrm{d}x.
\end{equation}

Applying the equality in \cite[Lemma 7.1]{GRS11} repeatedly to the sequence of quadruples in \eqref{thm1equ20} for $i$ going from 1 to $n-1$, we obtain that
\begin{equation}\label{thm2equ2}
f(g)=\int_{\prod_{s=1}^{n-1} \prod_{j=1}^s Y^s_j
E_{2n+1,n+1}
(\BA)}\int_{[\wt{W}_{n-1}']} \varphi(wxg) \psi^{-1}_{\wt{W}_{n-1}'}(w)\ \mathrm{d}w\mathrm{d}x.
\end{equation}

Next, take the Fourier expansion of $f$ along $E_{n,3n+2}$. Under the action of $\GL_1$, we get two kinds of Fourier coefficients corresponding to the two orbits of the dual of $[E_{n,3n+2}]$: the trivial one and the non-trivial one. Since we have assumed that $\CE_{\tau \otimes \sigma}$ has no nonzero Fourier coefficients attached to the partition $[(2n+1)1^{2n+1}]$, by a similar argument as in the proof of Lemma \ref{vanishinglem}, all the Fourier coefficients corresponding to the non-trivial orbit are identically zero. Hence, we obtain that
\begin{equation}\label{thm2equ7}
f(g)=\int_{\prod_{s=1}^{n-1} \prod_{j=1}^s Y^s_j
E_{2n+1,n+1}(\BA)} \int_{[E_{n,3n+2}]}
\int_{[\wt{W}_{n-1}']} \varphi(wxyg) \psi^{-1}_{\wt{W}_{n-1}'}(w)\ \mathrm{d}w\mathrm{d}x\mathrm{d}y.
\end{equation}
Then, applying the equality in \cite[Lemma 7.1]{GRS11} repeatedly to the sequence of quadruples in \eqref{thm1equ14}, one can see that
\begin{equation}\label{thm2equ8}
f(g)=\int_{\prod_{j=1}^{n-1} Y^n_j \prod_{s=1}^{n-1} \prod_{j=1}^s Y^s_j
E_{2n+1,n+1}(\BA)}
\int_{[\wt{W}_{n}']} \varphi(wxg) \psi^{-1}_{\wt{W}_{n}'}(w)\ \mathrm{d}w\mathrm{d}x.
\end{equation}

Carrying out the steps from \eqref{thm1equ18} to \eqref{thm1equ11},
applying the equality in \cite[Lemma 7.1]{GRS11} instead of \cite[Corollary 7.1]{GRS11}, we obtain that
\begin{equation}\label{thm2equ6}
f(g)=\int_{\prod_{s=n+1}^{2n-1} Y_s \prod_{j=1}^{2n-s-1} Y^s_j\prod_{k=1}^{n-1} \prod_{\ell=1}^k Y^k_{\ell}
E_{2n+1,n+1}
(\BA)}\int_{[N_{2n}']} \varphi^U(nxg) \psi^{-1}_{N_{2n}'}(n) \ \mathrm{d}n\mathrm{d}x,
\end{equation}
where $N_{2n}', U, \psi^{-1}_{N_{2n}'}$ are defined as in the proof of Proposition \ref{thm1},
$\varphi^U$ is the constant term of $\varphi$ along $U$.
Note that in \eqref{thm2equ6},
$\prod_{s=n+1}^{2n-1} Y_s \prod_{j=1}^{2n-s-1} Y^s_j \prod_{k=1}^{n-1} \prod_{\ell=1}^k Y^k_{\ell}
E_{2n+1,n+1}$ is equal to $V_{[(2n)^21^2,2]}Y \cap U^{-}$, where $U^{-}$ is the unipotent radical of the parabolic subgroup opposite to $P=MU$.
By a similar calculation as in \cite{JLZ13},
we deduce that
$$\varphi^U \in \CA(U(\BA)M(F)\bs\SO^{\delta}_{4n+2}(\BA))_{\lvert \cdot \rvert^{-\frac{1}{2}} \tau \otimes \sigma}.$$

For $t \in \BA^{\times}$, let
$D(t)=\begin{pmatrix}
tI_{2n} & 0 & 0\\
0 & I_2 & 0 \\
0 & 0 & t^{-1} I_{2n}
\end{pmatrix}$.
Then it is easy to see that $\omega D(t) \omega^{-1}=D(t^{-1})$, where $\omega$ is the Weyl element defined in \eqref{weylelement}.
Consider $f(D(t)g)$. Note that $f(g)$ has the form in \eqref{thm2equ6}. Conjugating $D(t)$ to the left,
by changing of variables on
$$\prod_{s=n+1}^{2n-1} Y_s \prod_{j=1}^{2n-s-1} Y^s_j \prod_{k=1}^{n-1} \prod_{\ell=1}^k Y^k_{\ell}
E_{2n+1,n+1}(\BA),$$
we get a factor $\lvert t \rvert_{\BA}^{-2n^2+1}$.
Since $\varphi^U \in \CA(U(\BA)M(F)\bs\SO^{\delta}_{4n+2}(\BA))_{\lvert \cdot \rvert^{-\frac{1}{2}} \tau \otimes \sigma}$,
we have
$$\varphi^U(D(t)nxg) = \delta_P(D(t))^{\frac{1}{2}} \lvert D(t) \rvert^{-\frac{1}{2}} \varphi^U(nxg)=\lvert t \rvert^{2n^2}_{\BA} \varphi^U(nxg).$$
Note that $P=MU$ is the parabolic subgroup with Levi $M \cong \GL_{2n} \times \SO_2^{\delta}$.
Therefore, one can get $f(D(t)g) = \lvert t \rvert_{\BA} f(g)$.
On the other hand,
$$f(D(t)g)=f(\omega D(t) g) = f(D(t^{-1})\omega g) = \lvert t^{-1} \rvert_{\BA} f(\omega g) =\lvert t \rvert^{-1}_{\BA} f(g).$$
Hence, $\lvert t \rvert_{\BA} f(g) =\lvert t \rvert^{-1}_{\BA} f(g)$. Since $t \in \BA^{\times}$, we get $f(g) \equiv 0$, which is a contradiction.

Therefore, the residual representation $\CE_{\tau \otimes \sigma}$ has a nonzero Fourier coefficient attached to the
partition $[(2n+1)1^{2n+1}]$. This completes the proof of the proposition.
\end{proof}

We remark that the above fleshes out a sketch given in \cite[Section 5]{GJR05}, with full details for each step and precise references as needed.

\begin{prop}[Part (5) of Theorem \ref{Main theorem of descent}]\label{thm3}
Every irreducible component of the twisted automorphic descent $\sigma_{\psi_{n, \alpha}}(\CE_{\tau \otimes \sigma})$ has
a nonzero Fourier coefficient attached to the partition $[(2n-1)1^2]$.
\end{prop}

\begin{proof}
Recall that $\sigma_{\psi_{n, \alpha}}(\CE_{\tau \otimes \sigma})$ is the completion space spanned by the cuspidal automorphic functions produced by the twisted descent, by Part (2) of Theorem \ref{Main theorem of descent}.
This twisted descent $\sigma_{\psi_{n, \alpha}}(\CE_{\tau \otimes \sigma})$ projects onto the decomposition of the cuspidal spectrum of $\SO^{\delta,\alpha}_{2n+1}$.
Let $\pi$ be any irreducible consituent in this decomposition.
Consider the following integral
\begin{equation}\label{thm3equ1}
\langle \varphi_{\pi}, {\xi}^{\psi_{n, \alpha}} \rangle
=\int_{[\SO^{\delta, \alpha}_{2n+1}]} \varphi_{\pi}(h)\ol{{\xi}^{\psi_{n, \alpha}}}(h)\ \mathrm{d}h,
\end{equation}
which is nonzero for some data $\varphi_{\pi} \in \pi$,
${\xi} \in \CE_{\tau \otimes \sigma}$,
since $\pi$ is an irreducible component
of $\sigma_{\psi_{n, \alpha}}(\CE_{\tau \otimes \sigma})$.

Assume that ${\xi} = \Res_{s=\frac{1}{2}} {E}(\cdot, s, \phi_{\tau \otimes \sigma})$,
then from \eqref{thm3equ1} we know that the following integral
is also nonzero for some choice of data:
\begin{equation}\label{thm3equ2}
\langle \varphi_{\pi}, {{E}(\cdot, s, \phi_{\tau \otimes \sigma})}^{\psi_{n, \alpha}} \rangle
=\int_{[\SO^{\delta, \alpha}_{2n+1}]} \varphi_{\pi}(h)\ol{{{E}(h, s, \phi_{\tau \otimes \sigma})}^{\psi_{n, \alpha}}}\
\mathrm{d}h.
\end{equation}
Then, by the unfolding in \cite[Proposition 3.7]{JZ14} (take $j=2n$, $\ell = n$ and $\beta=j-\ell=n$ there),
the non-vanishing of the integral in \eqref{thm3equ2}
implies the non-vanishing of ${\varphi_{\pi}}^{\psi_{{n-1}, {\alpha'}}}$, for some $\alpha' \in F^{\times}$.
Therefore, $\pi$ has a non-vanishing Fourier coefficient attached to the partition $[(2n-1)1^2]$.
This completes the proof of the proposition.
\end{proof}

Finally, Propositions \ref{thm2} and \ref{thm3} imply that the residual representation $\CE_{\tau \otimes \sigma}$ has a nonzero
Fourier coefficient attached to the partition $[(2n+1)(2n-1)1^2]$ by a similar argument as in the proofs of
\cite[Lemma 2.6]{GRS03} and
\cite[Lemma 3.1 and Proposition 3.3]{JL15}.
This proves Part (6) of Theorem \ref{Main theorem of descent}.
The analogues for orthogonal groups of \cite[Lemma 3.1 and Proposition 3.3]{JL15} will be discussed in \cite{L14}.

It is natural to ask about the relation of Part (6) of Theorem \ref{Main theorem of descent} with Conjecture 4.2 in \cite{J14}.
More precisely, one may note that for suitable pairs $(\alpha,\delta)$ the group $\SO^{\delta,\alpha}_{2n+1}$ is split,
and one may then ask for which such pairs the twisted automorphic descent
$\sigma_{\psi_{n,\alpha}}( \CE_{\tau\otimes\sigma})$ has no generic irreducible summands.
As we remarked in the Introduction, this is one of the technical key points to connect the
construction of twisted automorphic descents to the structure of more general global packets.
In Section 5, we will provide the complete theory for the case of $n=1$, which recovers the Jacquet-Langlands correspondence for $\PGL_2$.


\section{Relations with the Langlands Functorial Transfers}


We first give a proof for Part (4) of Theorem \ref{Main theorem of descent} and then discuss the relation between the construction of the
twisted automorphic descent given in this paper and the automorphic descent given by Ginzburg, Rallis and Soudry in \cite{GRS11}, which
is the relation with the corresponding Langlands functorial transfer.

\subsection{Part (4) of Theorem \ref{Main theorem of descent}}
Assume that the twisted automorphic descent $\sigma_{\psi_{n,\alpha}}( \CE_{\tau\otimes\sigma})$ is nonzero. By Part (2) of
Theorem \ref{Main theorem of descent}, $\sigma_{\psi_{n,\alpha}}( \CE_{\tau\otimes\sigma})$ is cuspidal on $\SO^{\delta,\alpha}_{2n+1}(\BA)$.
We write
$$
\sigma_{\psi_{n,\alpha}}( \CE_{\tau\otimes\sigma})
=
\oplus_i\pi_i
$$
where $\pi_i$ are irreducible cuspidal automorphic representations of $\SO^{\delta,\alpha}_{2n+1}(\BA)$. We have to show that
\begin{enumerate}
\item the irreducible summands are nearly equivalent, that is, their local components at almost all unramified local places are equivalent; and
\item the weak Langlands functorial transfer from $\SO^{\delta,\alpha}_{2n+1}$ to $\GL_{2n}$ takes $\pi_i$ to the given $\tau$.
\end{enumerate}
It is clear that (2) implies (1). As discussed in \cite{GGP12}, when $\alpha$ varies in the square classes of $F^\times$, the family
$\SO^{\delta,\alpha}_{2n+1}$ are pure inner $F$-rational forms of each other, and hence they share the Langlands dual group, which is $\Sp_{2n}(\BC)$.
The Langlands functorial transfer considered here is the one determined by the natural embedding of $\Sp_{2n}(\BC)$ into $\GL_{2n}(\BC)$.

Let $\pi$ be one of the irreducible summands in the decomposition of $\sigma_{\psi_{n,\alpha}}( \CE_{\tau\otimes\sigma})$. Write
$\pi=\otimes_v\pi_v$. Take $v$ to be a finite local place of $F$ where $\pi_v$, $\tau_v$, $\sigma_v$ and the other relevant data in the
construction of the descent are unramified. Hence $\pi_v$ is an irreducible unramified representation of the $F_v$-split $\SO_{2n+1}(F_v)$,
occurring as the unramified constituent in the Jacquet module $\displaystyle{J_{\psi_{n,\alpha}}(\Ind_{Q_{2n}(F_v)}^{L_{n,\alpha}(F_v)}\tau_v|\det(\cdot)|_v^{\frac{1}{2}}\otimes \sigma_v)}$
by the local-global relation of the construction of the descent.

On the other hand, since $\tau_v$ is unramified, generic and self-dual, it follows that $\tau_v$ is the fully induced representation
$$
\mu_1\times \cdots \times \mu_n\times \mu_n^{-1}\times \cdots \times \mu_1^{-1},
$$
where $\mu_i$'s are unramified characters.
Also, $\sigma_v$ is an unramified character $\mu_0$ on $F_v^\times$ or a trivial representation on $\SO^\delta_2(F_v)$,
depending on whether $J_\delta$ is split or not. Now using Proposition \ref{Jacquet module: list}, we obtain that
the Satake parameter of $\pi_v$ is transferred to the Satake parameter of $\tau_v$. This finishes the proof of Part (4) of
Theorem \ref{Main theorem of descent}.

\subsection{On functorial relations}
When an irreducible unitary cuspidal representation $\tau$ of $\GL_{2n}(\BA)$ has the property that the exterior square
$L$-function
$L(s,\tau,\wedge^2)$ has a pole at $s=1$, the automorphic descent of Ginzburg-Rallis-Soudry in \cite{GRS11} constructs an
irreducible
generic cuspidal automorphic representation $\pi_0$ of $F$-split $\SO_{2n+1}(\BA)$.
The irreducibility of their descent was proved in
\cite{JngS03}.

On the other hand, when $\tau$ has the additional property that the central value $L(\frac{1}{2},\tau\otimes\sigma)$ is nonzero,
the theory of the twisted automorphic descents produces additional irreducible cuspidal automorphic representations $\pi$ of the split form of $\SO_{2n+1}$.
These representations arise at values of $\alpha$ where $\SO^{\delta,\alpha}$ happens to be split, and are nearly equivalent to $\pi_0$ by Part (4) of Theorem 1.2.
Also, they belong to the same
global $L$-packet with the global Arthur parameter $(\tau,1)$ (\cite{A}).

Moreover, for the $\alpha$'s, which make $\SO^{\delta,\alpha}_{2n+1}$ nonsplit over $F$, the
twisted automorphic descents $\sigma_{\psi_{n,\alpha}}( \CE_{\tau\otimes\sigma})$ produce irreducible cuspidal
automorphic representations $\pi$ of $\SO^{\delta,\alpha}_{2n+1}(\BA)$, which are still nearly equivalent to $\pi_0$ and
hence are expected to be in the same global Vogan packet as $\pi_0$.
By a careful investigation one ought to be able to verify the expectation regarding global Vogan packets,
and we refer to \cite{JZ15} for more detailed discussion on the general theory related to this issue.
A complete story is discussed in Section 5 below for the special case of $n=1$.


\section{Explicit Theory of Twisted Automorphic Descents for the Case of $n=1$}


In this section, we will give an explicit description on the twisted automorphic descent
$\sig_{\psi_{{n},\alpha}}(\CE_{\tau\otimes\sig})$
in the case $n=1$, which recovers the global Jacquet-Langlands correspondence between $\GL_2$ and its inner forms.

In this case, we have that $H^{\delta}=\SO^{\delta}_{6}$ and we take $\ell=1$.
The quadratic subspace $W_{1}$ as defined in Section 1.2 has a basis $\{e_{2},e^{(1)}_{0},e^{(2)}_{0},e_{-2}\}$.
Recall that $P_{1}=M_{1}N_{1}$ is the standard parabolic subgroup with the Levi $M_{1}=\GL_{1}\times\SO(W_{1})$,
and
$$
N_{1}=\left\{n_{1}(x)=\begin{pmatrix}
1&x&-\frac{1}{2}q(x,x)\\
&I_{4}&x'\\
&&1
\end{pmatrix}\ \bigg|\  x\in W_{1}\right\},
$$
where $q(\cdot,\cdot)$ is the quadratic form of $W_{1}$.
Since the Witt index of $W_{1}$ is greater than zero,
$\SO(W_{1})$ acts transitively on the set of vectors of the same length.
Thus, for every anisotropic vector $w_{0}$ in $W_{1}$,
we may choose $y_{\alpha}=e_{2}-\frac{\alpha}{2}e_{-2}$
as a standard anisotropic vector for some $\alpha \in F^{\times}$.
Then we focus on the character $\psi_{1,\alpha}$ of $N_{1}(F)\bks N_{1}(\BA)$ and
the stabilizer $L_{1,\alpha} =\SO(J_{\delta,\alpha})$.
If $J_{\delta,\alpha}$ is non-split, then $L_{1,\alpha}$ is isomorphic to ${\rm PD}^{\times}_{\delta,\alpha}$,
where ${\rm D}_{\delta,\alpha}$ is the quaternion algebra $(\frac{-\delta,-\alpha}{F})$.
Note that ${\rm D}_{\delta,\alpha}$ is uniquely determined by the coset of $\alpha $ in
$ F^{\times}/{\rm Nm}(F(\sqrt{-\delta})^{\times})$

Recall that $\sig$ is a character of $\SO^{\delta}_{2}(\BA)$ and $\tau$ is an irreducible unitary cuspidal
automorphic representation of $\GL_{2}(\BA)$. By Proposition~\ref{jlz13},
the residual representation $\CE_{\tau\otimes\sig}$ is nonzero
 if and only if $L(s,\tau,\wedge^{2})$ has a pole at $s=1$ and $L(\frac{1}{2},\tau\times\sig)\ne 0$.
Note that  $L(s,\tau,\wedge^{2})$ has a pole at $s=1$ if and only if $\tau$ has a trivial central character.
Throughout this section, we assume that $\tau$ has a trivial central character and hence $\tau$ can be regarded as
a representation of $\PGL_{2}(\BA)$.

Let $\pi$ be an irreducible cuspidal automorphic representation of $L_{1,\alpha}(\BA)$ and $\varphi$ an automorphic form in $\pi$.
Define the global zeta integral by
\begin{equation}
\CZ(s,\phi_{\tau\otimes\sig},\varphi,\psi_{1,\alpha})
=\int_{[L_{1,\alpha}]}\varphi(h)E^{\psi_{1,\alpha}}(h,s,\phi_{\tau\otimes\sig}) \ud h,
\end{equation}
where $E^{\psi_{1,\alpha}}(\cdot)$ is the $\psi_{1,\alpha}$-Fourier coefficient  of the Eisenstein series
$E(h,s,\phi_{\tau\otimes\sig})$ defined in~\eqref{Gelfand-Graev}. This global zeta integral is a special case of those
considered in \cite{GPSR97} and \cite{JZ14}.

We recall some notation in \cite{JZ14} before evaluating this global zeta integral.
Let $\eps$ be the minimal representative in the Weyl group of $\SO^{\delta}_{6}$ corresponding to the open cell in $P_{2}\bks \SO^{\delta}_{6}/P_{1}$, that is,
\begin{equation*}
\eps=\begin{pmatrix}
0&1&&&0&0\\
0&0&&&0&1\\
&&1&0&&\\
&&0&-1&&\\
1&0&&&0&0\\
0&0&&&1&0
\end{pmatrix}\in \SO^{\delta}_{6},
\end{equation*}
then the complement of $N^{\eps}_{1}=N_{1}\cap \eps^{-1} P_{2}\eps$ in $N_{1}$ is
$$
\bar{N}^{\eps}_{1}:=\{n_{1}((x,y,z,0))\mid (x,y,z,0)\in W_{1}\}.
$$
Define
$$
\CJ(s,\phi_\alpha^{W})(h)=\int_{N^{\eps}_{1}(\BA)\bks N_{1}(\BA)}
\lam_{s}\phi_\alpha^{W}(\eps n h)\psi^{-1}_{1,\alpha}(n)\ud n,
$$
with
$$
\phi_\alpha^{W}(h)=\int_{[N_{1,2}]}\phi_{\tau \otimes \sigma}(u_{1,2}(x)h)\psi^{-1}(\frac{\alpha}{2}x)\ud x.
$$
where $N_{1,2} = \{u_{1,2}(x)=\begin{pmatrix}
1 & x & & & &\\
0 & 1 &&&&\\
&& 1 & &&\\
&&& 1 &&\\
&&&& 1 &-x\\
&&&&&1
\end{pmatrix}\}$. 
Since the adjoint action of $\eps$ on the subgroup $\SO_{2}^{\delta}$ of the Levi subgroup of $P_{2}$  is  the inverse on $\SO^{\delta}_{2}$,
we obtain that
\begin{align*}
\CJ(s, \phi_\alpha^{W})(xh)=\sig^{-1}(x)\CJ(s, \phi_\alpha^{W})(h)
\end{align*}
where $x\in \SO_{2}^{\delta}$ is embedded into $\SO^{\delta}_{6}$ via the Levi subgroup.

Following the unfolding of the global zeta integral for the general case considered in \cite{GPSR97} and \cite{JZ14},
we obtain
\begin{align}
\CZ(s,\phi_{\tau\otimes\sig},\varphi,\psi_{1,\alpha})
=&\int_{\SO^{\delta}_{2}(\BA)\bks \SO^{\delta}_{3}(\BA)}
\CJ(s, \phi_\alpha^{W})(\epsilon   h)
\CP_{\sig^{-1}}(\varphi)(h) \ud h,
\end{align}
which has an eulerian product decomposition, where
$$
\CP_{\sig^{-1}}(\varphi)(h):=\int_{[\SO^{\delta}_{2}]}\varphi(xh)\sig^{-1}(x)\ud x.
$$

We denote by ${\rm JL}$ the Jacquet-Langlands correspondence (\cite{J-L})
between the set $\CA_{cusp}(L_{1,\alpha})$ and the set $\CA_{cusp}(\PGL_{2})$,
where $\CA_{cusp}(L_{1,\alpha})$ is the set of irreducible automorphic representations of $L_{1,\alpha}(\BA)$ of infinite dimension
and $\CA_{cusp}(\PGL_2)$ is the set of irreducible cuspidal automorphic representations of $\PGL_2(\BA)$.
 Referring to Theorem 2 \cite{W85a}, the period integral $\CP_{\sig^{-1}}(\varphi)$ is nonzero on $\pi$ if and only if
$$
 L(\frac{1}{2},{\rm JL}(\pi)\times\sig^{-1})\ne 0
 \ \ \text{ and }\ \
 \Hom_{\SO^{\delta}_{2}(F_{v})}( \pi_{v},\sig_{v})\ne 0 \text{ for all }v.
 $$
 Note that the dimension of the Hom space is at most one.
In addition, by the theorem of Tunnell-Saito in \cite{T83, S93}, an equivalent statement (see Section V \cite{W91}) for the local period condition is that
$$\SO^{\delta,\alpha}_{3} \text{ is ramified at a place $v$ of $F$ if and only if }
\varepsilon(\frac{1}{2},{\rm JL}(\pi)_{v}\times\sig_{v},\psi_v)\eta_{E,v}(-1)=-1,$$
where $\eta_{E}$ is a quadratic character of $F^{\times}\bks\BA^{\times}$ associated to the quadratic extension $E:=F(\sqrt{-\delta})$ via the class field theory.
 Note that $\SO^{\delta,\alpha}_{3}\cong \PD^\times_{\delta,\alpha}$ is ramified at a local place $v$, meaning that $\SO^{\delta,\alpha}_{3}$ is not split at $v$.
Therefore for a pure tensor product vector $\varphi=\otimes_v\varphi_v$ in $\pi=\otimes_v\pi_v$
\begin{equation}
\CP_{\sig^{-1}}(\varphi)(h)=C_{0}\prod_{v}\ell_{v}(\pi_{v}(h_v)\varphi_{v}),
\end{equation}
where $\ell_{v}$ is a nonzero element in $\Hom_{L_{1,\alpha}(F_{v})}( \pi_{v},\sig_{v})$, and $C_{0}$ is $1$ if $L(\frac{1}{2},{\rm JL}(\pi)\times\sig)\ne 0$ and it is zero otherwise.

Let $S$ be a finite set of places containing all archimedean places, such that all data is unramified over places outside $S$.
Then,
\begin{align}
\CZ(s,\phi_{\tau\otimes\sig},\varphi_{\pi},\psi_{1,\alpha})
=&C_{0}\frac{L^{S}(s+\frac{1}{2},\pi\times\tau)}{L^{S}(s+1,\tau\times\sig)L^{S}(2s+1,\tau,\wedge^{2})}
\prod_{v\in S}\CZ_{v}(s,\phi_{v},\varphi_{v},\psi_{1,\alpha})\label{eq:GPSR},
\end{align}
where the local zeta integrals at the ramified local places are given by
$$
\CZ_{v}(s,\phi_{v},\varphi_{v},\psi_{1,\alpha})=\int_{\bar{N}^{\eps}_{1}(F_{v})}
\CJ_{v}(s, \phi_{\alpha,v}^{W})(h)\ell_{v}(\pi_{v}(h)\varphi_{v})\ud h.
$$
The first main result in the case $n=1$ of the general theory of twisted automorphic descents (as displayed in the previous sections and
also in \cite{JZ15}) can be formulated as follows.

\begin{thm}\label{thm:n=1}
With notation as above, assume that $\tau$ is of trivial central character.
Given $\delta$ and $\sig$, if the residual representation $\CE_{\tau\otimes\sig}$ is not identically zero,
then the following hold.
\begin{enumerate}
\item[(1)]
The set of $\alpha\in F^\times$ such that the $\psi_{1,\alpha}$-Fourier coefficient $\CE^{\psi_{1,\alpha}}_{\tau\otimes\sig}$ is not identically zero is a single coset $\alpha_0\cdot {\rm Nm}(E^{\times})$.
\item[(2)] The twisted automorphic descent $\sig_{\psi_{1,\alpha}}(\CE_{\tau\otimes\sig})$ is irreducible and has the property that
$$
{\rm JL} (\sig_{\psi_{1,\alpha}}(\CE_{\tau\otimes\sig}))\cong \tau.
$$
\item[(3)] The norm class $\alpha_{0}$ is determined by the property that
$\RD_{\delta,\alpha_{0}}$ is ramified at a place $v$ of $F$ if and only if
$\varepsilon(\frac{1}{2},\tau_{v}\times\sig_{v},\psi_v)\eta_{E,v}(-1)=-1$.
\end{enumerate}
\end{thm}

\begin{proof}
By the definition of the $\psi_{1,\alpha}$-Fourier  coefficients and Proposition~\ref{Jacquet module: list}
and~\ref{cuspidality}, the $\psi_{1,\alpha}$-Fourier coefficient $\CE^{\psi_{1,\alpha}}_{\tau\otimes\sig}(\phi)(h)$ on
$L_{1,\alpha}(\BA)$  is in $L^{2}_{cusp}(L_{1,\alpha}(F)\bks L_{1,\alpha}(\BA))$.
Applying the spectral decomposition on $L^{2}_{cusp}(L_{1,\alpha}(F)\bks L_{1,\alpha}(\BA))$, we have
\begin{equation}\label{eq:spectral-de}
\CE^{\psi_{1,\alpha}}_{\tau\otimes\sig}(\phi)=\sum_{\pi\in\CA_{cusp}(L_{1,\alpha})}
\sum_{\varphi_{\pi}\in B(\pi)}\langle\CE^{\psi_{1,\alpha}}_{\tau\otimes\sig}(\phi),\varphi_{\pi}\rangle\varphi_{\pi}
\end{equation}
where  $\langle\cdot,\cdot\rangle$ is the inner product in $L^{2}_{cusp}(L_{1,\alpha}(F)\bks L_{1,\alpha}(\BA))$ and
$B(\pi)$ is an orthonormal basis of the space $\pi$.

Let us compute the spectral coefficient $\langle\CE^{\psi_{1,\alpha}}_{\tau\otimes\sig}(\phi),\varphi_{\pi}\rangle$.
It is clear that the integration domain $N_{1}(F)\bks N_{1}(\BA)$ is compact and $\varphi_{\pi}$ is rapidly decreasing,
so, we can switch the order of taking the residue and taking the integration, and obtain
\begin{align*}
\Res_{s=\frac{1}{2}}\CZ(s,\phi_{\tau\otimes\sig},\bar{\varphi}_{\pi},\psi_{1,\alpha})
=&\Res_{s=\frac{1}{2}}\int_{[L_{1,\alpha}]}
E^{\psi_{1,\alpha}}(\phi_{\tau\otimes\sig},s)(h)\bar{\varphi}_{\pi}(h)\ud h\\
=&\int_{[L_{1,\alpha}]}\CE^{\psi_{1,\alpha}}_{\tau\otimes\sig}(\phi)(h)\bar{\varphi}_{\pi}(h)\ud h\\
=&\langle\CE^{\psi_{1,\alpha}}_{\tau\otimes\sig}(\phi),\varphi_{\pi}\rangle.
\end{align*}
It follows that the Fourier coefficient $\CE^{\psi_{1,\alpha}}_{\tau\otimes\sig}(\phi)$ is not identically zero
if and only if the residue $\Res_{s=\frac{1}{2}}\CZ(s,\cdot)$ is not zero for some choice of $\pi$ and some choice of $\varphi_{\pi}\in\pi$.

By Part (3) of Theorem \ref{Main theorem of descent}, $\Res_{s=\frac{1}{2}}\CZ(s,\cdot)$ is zero unless taking $\pi$ such that ${\rm JL}(\pi)\cong \tau$.
Let us consider $\pi$ with ${\rm JL}(\pi)\cong \tau$.
Recall that the infinite eulerian products for $L^S(s,\tau\times\sig)$ and $L^S(s,\tau,\wedge^2)$ converge absolutely for $\Re(s)>1$ following by Jacquet and Shalika in \cite{JS81,JS90} and then are nonvanishing at $\Re(s)>1$.
Since $\CZ(s,\cdot)$ has at most a simple pole at $s=\frac{1}{2}$ and $L^S(s+\frac{1}{2},\pi\times\tau)$ has a pole at $s=\frac{1}{2}$, by Equation~\ref{eq:GPSR},
$\CZ_{v}(s,\phi_{\tau_{v}\otimes\sig_{v}},\varphi_{v},\psi_{1,\alpha})$ for $s\in S$ is holomorphic at $s=\frac{1}{2}$.
If $\Hom_{\SO^{\delta}_{2}(F_{v})}( \pi_{v},\sig_{v})\ne 0$, then there exists a choice of data $\phi_{\tau_v\otimes\sig_v}$ and  $\varphi_v$ such that the local zeta integral $\CZ(s,\dot)$ is not zero at $\Re(s)=\frac{1}{2}$.
The argument is similar to the one in Sections 6 and 7 \cite{S93}. We omit the details here.
Also a more general situation of the same nature has been discussed in \cite[Appendix A]{JZ15}.

Hence, according to Equation~\eqref{eq:GPSR}, $\Res_{s=\frac{1}{2}}\CZ(s,\cdot)$ is nonzero if and only if
\begin{enumerate}
\item[(i)] $L^{S}(s+\frac{1}{2},\pi\times \tau)$ has a pole at $s=\frac{1}{2}$;
\item[(ii)] $L(\frac{1}{2},{\rm JL}(\pi)\times\sig)\ne 0$;
\item[(iii)] $L_{1,\alpha}$ is ramified at $v$ if and only if
$\varepsilon(\frac{1}{2},{\rm JL}(\pi)_{v}\times\sig_{v},\psi_v)\eta_{E,v}(-1)=-1$.
\end{enumerate}
By the equality of $L$-functions:
$$
L^{S}(s,\pi\times\tau)=L^{S}(s,{\rm JL}(\pi)\times\tau),
$$
it is clear that
$L^{S}(s+\frac{1}{2},{\rm JL}(\pi)\times\tau)$ has a pole at $s=\frac{1}{2}$ if and only if ${\rm JL}(\pi)=\tau$. Hence
Condition~(i) holds if and only if ${\rm JL}(\pi)=\tau$.

Further, by Proposition \ref{jlz13}, $\CE_{\tau\otimes\sig} \ne 0$ implies $L(\frac{1}{2},\tau\times\sig)\ne 0$.
Hence when ${\rm JL}(\pi)=\tau$ Condition~(ii) holds if $\CE_{\tau\otimes\sig} \ne 0$. Finally,
the twisted automorphic descent $\sig_{\psi_{1,\alpha}}(\CE_{\tau\otimes\sig})$ is zero unless the group
$L_{1,\alpha}$ is determined by Condition~(iii).
In such a case, by the uniqueness of the local Bessel model, if $\sig_{\psi_{1,\alpha}}(\CE_{\tau\otimes\sig})$ is nonzero, there is a unique cuspidal automorphic representation $\pi$ such that $\Res_{s=\frac{1}{2}}\CZ(s,\cdot)$ is nonzero and then $\sig_{\psi_{1,\alpha}}(\CE_{\tau\otimes\sig})$ is irreducible.
By Proposition~\ref{thm2}, the twisted automorphic descent
$\sig_{\psi_{1,\alpha}}(\CE_{\tau\otimes\sig})$ is nonzero for some $\alpha$.
On the other hand, if $\alpha_0$ are congruent to $\alpha$ modulo ${\rm Nm}(E^{\times})$, then $L_{1,\alpha_0}$ is isomorphic to $L_{1,\alpha}$ and $\sig_{\psi_{1,\alpha}}(\CE_{\tau\otimes\sig})$ is also an automorphic representation of $L_{1,\alpha_0}(\BA)$,
which satisfies Conditions (i)--(iii).
Thus the twisted descent $\sig_{\psi_{1,\alpha_0}}(\CE_{\tau\otimes\sig})$ is nonzero.
Therefore, there is a unique $\alpha_{0}$ modulo ${\rm Nm}(E^{\times})$ such that
$\sig_{\psi_{1,\alpha}}(\CE_{\tau\otimes\sig})$ is not identically zero
and ${\rm JL}(\sig_{\psi_{1,\alpha}}(\CE_{\tau\otimes\sig}))=\tau$ by the spectral decomposition~\eqref{eq:spectral-de}.
\end{proof}

\begin{rmk}
It is clear from the above theorem that the projective quaternion group $\SO^{\delta,\alpha}_{3}$ is unique modulo
${\rm Nm}(E^{\times})$.
However, the $F$-rational orbits of the characters $\psi_{1,\alpha}$ with the stabilizer $\PD_{\delta,\alpha}^{\times}$ are in one-to-one
correspondence with the square classes
$${\rm Nm}(E^{\times})/F^{\times 2}.$$
\end{rmk}

The following theorem asserts that by choosing suitable $\delta$ and $\sig$ involved in the construction of
the twisted automorphic descents,
we are able to obtain all infinite dimensional automorphic representations of ${\rm PD}^{\times}(\BA)$, which
is the second main result for this special case. We refer to \cite{JZ15} for a treatment of the general situation.

\begin{thm}\label{thm:construct}
Let ${\rm D}$ be a quaternion algebra containing a quadratic extension $F(\sqrt{-\delta})$ of $F$.
For any given infinite dimensional irreducible automorphic representation $\pi$ of ${\rm PD}^{\times}(\BA)$,
there exists a character $\sig$ of $\SO^{\delta}_{2}(\BA)$ such that the residual representation $\CE_{{\rm JL}(\pi)\otimes\sig}$ is nonzero and
an element $\alpha\in F^\times$ such that ${\rm D}_{\delta,\alpha}\cong {\rm D}$ and $\sig_{\psi_{1,\alpha}}(\CE_{{\rm JL}(\pi)\otimes\sig})\cong\pi$.
Moreover, any such $\sig$ satisfies the period condition that $\CP_{\sig}(\pi)\neq 0$ or equivalently the $L$-function condition and the
local period condition that
\begin{equation}\label{eq:sig}
L(\frac{1}{2},{\rm JL}(\pi)\times\sig) \ne 0
\ \ \text{ and }\ \
\Hom_{\SO^{\delta}_{2}(F_{v})}(\pi_{v},\sig^{-1}_{v})\ne 0 \text{ for all }v.
\end{equation}
\end{thm}

\begin{proof}
Consider $\SO^{\delta}_{2}(\BA)$ as a subgroup of ${\rm PD}^{\times}(\BA)$.
Since $\SO^{\delta}_{2}(F)\bks\SO^{\delta}_{2}(\BA)$ is compact and the space $\pi$ is an $L^{2}$-space, the restriction of $\pi$ onto the subgroup $\SO^{\delta}_{2}(F)\bks\SO^{\delta}_{2}(\BA)$ is semi-simple and nonzero.
By the spectral decomposition,
we can choose a character $\sig$ of $\SO^{\delta}_{2}(\BA)$ such that $\CP_{\sig}(\pi)\neq 0$.
By \cite{W85a}, it is equivalent that $L(\frac{1}{2},\pi\times\sig)\ne 0$ and $\Hom_{\SO^{\delta}_{2}(F_{v})}(\pi_{v},\sig^{-1}_{v})\ne 0$ for all places $v$.

Since
$L(\frac{1}{2},{\rm JL}(\pi)\times\sig)=L(\frac{1}{2},\pi\times\sig)\ne 0$ and ${\rm JL}(\pi)$ is of trivial central character, the residual representation $\CE_{{\rm JL}(\pi)\otimes\sig}$ is not identically zero.
By Theorem~\ref{thm:n=1}, there exists a unique $\SO^{\delta,\alpha}_{3}$
such that ${\rm JL}(\sig_{\psi_{1,\alpha}}(\CE_{{\rm JL}(\pi)\otimes\sig}))={\rm JL}(\pi)$.

It is enough to show that $\SO^{\delta,\alpha}_{3}$ is isomorphic to $\PD^{\times}$.
By the proof of  Theorem~\ref{thm:n=1}, $\CP_{\sig}(\sig_{\psi_{1,\alpha}}(\CE_{{\rm JL}(\pi)\otimes\sig}))\ne 0$.
Referring to Condition (iii) in the proof of Theorem~\ref{thm:n=1},
both $\SO^{\delta,\alpha}_{3}$ and $\PD^{\times}$ are ramified at $v$ if and only if
$\varepsilon(\frac{1}{2},{\rm JL}(\pi)_{v}\times\sig_{v},\psi_v)\eta_{E,v}(-1)=-1$.
Then $\SO^{\delta,\alpha}_{3}\cong\PD^{\times}$
 and $\sig_{\psi_{1,\alpha}}(\CE_{{\rm JL}(\pi)\otimes\sig})=\pi$.
\end{proof}
\begin{rmk}
By uniqueness of this model, the restriction of $\pi$ onto $\SO^{\delta}_{2}(F)\bks\SO^{\delta}_{2}(\BA)$ is multiplicity-free.
Since $\pi$ is of infinite dimension and $\sig$ is of dimension one, by the spectral decomposition as in the proof of Theorem~\ref{thm:construct}, there are infinitely many choices of $\sig$ satisfying \eqref{eq:sig}.
\end{rmk}

\end{document}